\documentclass{amsart}

\usepackage{mathtools}
\usepackage{float}
\usepackage{amsmath}
\usepackage{amsfonts}
\usepackage{amssymb}
\usepackage{graphicx}
\usepackage{hyperref}
\usepackage{mathrsfs}
\usepackage{pb-diagram}
\usepackage{epstopdf}
\usepackage{amsmath,amsfonts,amsthm,enumerate,amscd,latexsym,curves}
\usepackage{bbm}
\usepackage[mathscr]{eucal}
\usepackage{epsfig,epsf,xypic,epic}
\usepackage{caption}
\usepackage{subcaption}
\usepackage{tikz}
\usepackage{geometry}
\usepackage{enumitem}
\usetikzlibrary{matrix,arrows,decorations.pathmorphing}
\usetikzlibrary{cd}
\usepackage{tikz-cd}

\usetikzlibrary{matrix,arrows,decorations.pathmorphing}
\theoremstyle{theorem}
\newtheorem{theorem}{Theorem}[section]

\newtheorem{lemma}[theorem]{Lemma}
\newtheorem{corollary}[theorem]{Corollary}
\newtheorem{proposition}[theorem]{Proposition}

\theoremstyle{definition}
\newtheorem{definition}[theorem]{Definition}
\newtheorem{question}[theorem]{Question}

\newtheorem{example}[theorem]{Example}

\newtheorem{remark}[theorem]{Remark}

\numberwithin{equation}{section}

\newcommand{\Int}{\mathrm{Int}}

\newcommand{\Span}{\mathrm{Span}}

\newcommand{\Hom}{\mathrm{Hom}}
\newcommand{\Ext}{\mathrm{Ext}}
\newcommand{\trop}{\mathrm{trop}}

\newcommand{\dpsum}[2]{\displaystyle{\sum_{#1}^{#2}}}
\newcommand{\pd}[2]{\frac{\partial #1}{\partial #2}}

\newcommand{\inner}[1]{\langle #1\rangle}

\newcommand{\bb}[1]{\mathbb{#1}}
\newcommand{\cu}[1]{\mathcal{#1}}
\newcommand{\til}[1]{\widetilde{#1}}

\newcommand{\ol}[1]{\overline{#1}}

\newcommand{\mf}[1]{\mathfrak{#1}}
\newcommand{\ul}[1]{\underline{#1}}

\def\bR{\mathbb{R}}
\def\bC{\mathbb{C}}
\def\bP{\mathbb{P}}
\def\bL{\mathbb{L}}

\def\vol{\text{\rm vol}}
\def\ev{\text{\rm ev}}
\def\Hom{\text{\rm Hom}}

\begin{document}
	
	\title[]{Lagrangian multi-sections and their toric equivariant mirror}

	\author[Oh]{Yong-Geun Oh}
\address{Center for Geometry and Physics, Institute for Basic Science (IBS),
79 Jigok-ro 127beon-gil, Nam-gu, Pohang, Gyeongsangbuk-do, Korea 37673}
\email{yongoh1@postech.ac.kr}

	\author[Suen]{Yat-Hin Suen}
\address{Center for Geometry and Physics, Institute for Basic Science (IBS),
79 Jigok-ro 127beon-gil, Nam-gu, Pohang, Gyeongsangbuk-do, Korea 37673}
    \curraddr{School of Mathematics, Korea Institute for Advanced Study,
    85 Hoegiro Dongdaemun-gu, Seoul, Korea 02455}
	\email{yhsuen@kias.re.kr}

	\date{\today}

	\begin{abstract}
		The SYZ conjecture suggests a folklore that ``Lagrangian multi-sections are mirror to holomorphic vector bundles". In this paper, we prove this folklore for Lagrangian multi-sections inside the cotangent bundle of a vector space, which are equivariantly mirror to complete toric varieties by the work of Fang-Liu-Treumann-Zaslow. We also introduce the \emph{Lagrangian realization problem}, which asks whether one can construct an unobstructed Lagrangian multi-section with asymptotic conditions prescribed by a tropical Lagrangian multi-section. We solve the realization problem for tropical Lagrangian multi-sections over a complete 2-dimensional fan that satisfy the so-called $N$-generic condition. As an application, we show that every rank 2 toric vector bundle on the projective plane is mirror to a Lagrangian multi-section.
	\end{abstract}

	\maketitle
	
	\tableofcontents
	
	\section{Introduction}
	
	In \cite{CCC_HMS}, Fang-Liu-Treumann-Zaslow proved the \emph{equivariant homological mirror symmetry} of complete toric variety $X_{\Sigma}$ via \emph{coherent-constructible correspondence} (abbreviated as CCC) \cite{CCC} $$\kappa:\cu{P}erf_T(X_{\Sigma})\xrightarrow{\sim} Sh_{cc}(M_{\bb{R}},\Lambda_{\Sigma})$$
	and \emph{microlocalization} \cite{NZ, Nadler} $$\mu_{\Sigma}:Sh_{cc}(M_{\bb{R}},\Lambda_{\Sigma})\xrightarrow{\sim}\cu{F}uk(Y,\Lambda_{\Sigma}).$$
	Here, $Y:=T^*M_{\bb{R}}=M_{\bb{R}}\times N_{\bb{R}}$ and
	\begin{itemize}
		\item $\cu{P}erf_T(X_{\Sigma})$ is the category of equivariant perfect complexes.
		\item $Sh_{cc}(M_{\bb{R}},\Lambda_{\Sigma})$ is the dg-category of complexes of sheaves that are cohomologically constructible, compactly supported and with singular support laying in the conical Lagrangian subset
		$$\Lambda_{\Sigma}:=\bigcup_{\tau\in\Sigma}(\tau^{\perp}+M)\times(-\tau).$$
		\item $\cu{F}uk(Y,\Lambda_{\Sigma})$ is the Fukaya category generated by exact Lagrangian submanifolds with asymptotic condition prescribed by $\Lambda_{\Sigma}$.
	\end{itemize}
	Furthermore, they realized the functor $\mu_{\Sigma}\circ\kappa:\cu{P}erf_T(X_{\Sigma})\to \cu{F}uk(Y,\Lambda_{\Sigma})$ as a $T$-duality suggested by the celebrated SYZ program \cite{SYZ}. Namely, toric line bundles on $X_{\Sigma}$ correspond to Lagrangian sections of the projection map $p_{N_{\bb{R}}}:Y\to N_{\bb{R}}$ via SYZ transform. Moreover, as $\cu{P}erf_T(X_{\Sigma})$ is generated by line bundles, one can immediately deduce that $\cu{F}uk(Y,\Lambda_{\Sigma})$ is generated by Lagrangian sections of $p_{N_{\bb{R}}}:Y\to N_{\bb{R}}$.
	
	Despite the beauty of such homological mirror symmetry (abbreviated as HMS) statement, it doesn't tell us how geometric objects correspond to each other due to the abstract algebraic construction of the Fukaya category. A question that we want to address in this paper is the following question.
	
	\begin{question}
		What kind of objects in $\cu{F}uk(Y,\Lambda_{\Sigma})$ are mirror to toric vector bundles?
	\end{question}
	
	Before giving a symplecto-geometric answer to this question, Treumann has already provided a sheaf-theoretic condition to determine when a perfect complex on $X_{\Sigma}$ is quasi-isomorphic to a 1-term complex, that is, a single toric vector bundle.
	
	\begin{theorem}[=Theorem 1.9 in \cite{Morse_theory_TVB}]
		Let $X_{\Sigma}$ be a complete toric variety. Let $\cu{E}^{\bullet}\in\cu{P}erf_T(X_{\Sigma})$ and $\kappa(\cu{E}^{\bullet})\in Sh_{cc}(M_{\bb{R}},\Lambda_{\Sigma})$ be the corresponding complex of constructible sheaves. Then the following statements are equivalent:
		\begin{itemize}
			\item [(a)] $\cu{E}^{\bullet}$ is quasi-isomorphic to a 1-term complex concentrated at degree 0.
			\item [(b)] For any $m\in M$, $\sigma\in\Sigma_{max}$ and $\xi\in\Int(\sigma)$, the microlocal stalk
			$$\mu_{m,-\xi}(\kappa(\cu{E}^{\bullet}))\in Sh(pt)$$
			is concentrated at degree 0.
		\end{itemize}
	\end{theorem}
	
	To answer Question 1.1, we borrow an idea from the SYZ proposal and other works \cite{LYZ, CS_SYZ_imm_Lag}. As Lagrangian sections are mirror to holomorphic line bundles, \emph{Lagrangian multi-sections} should be mirror to holomorphic vector bundles. By a Lagrangian multi-section we mean a (graded but not necessarily exact) Lagrangian immersion $i:\til{L}\to Y$ that has at most finitely many double point singularities and the composition $p_{N_{\bb{R}}}\circ i:\til{L}\to N_{\bb{R}}$ is a branched covering map (Definition \ref{def:branced_covering} and \ref{def:LMS}).
	We denote by $L \subset Y$ the image $i(\widetilde L)$.
	
	Note that having a non-trivial branching set here is important as $N_{\bb{R}}$ is contractible. Otherwise, all multi-sections are simply disjoint union of sections, which can only be mirror to a direct sum of line bundles. To a Lagrangian multi-section $\bb{L}$, we can associate a \emph{canonical grading} so that the degree of intersection between $\bb{L}$ and a generic fiber $F_{\xi}\subset Y$ of $p_{N_{\bb{R}}}:Y\to N_{\bb{R}}$ is always $n$, the dimension of $M_{\bb{R}}$. A Lagrangian multi-section equipped with such grading is called \emph{canonically graded}. However, due to the non-exactness of general multi-sections, we need to pass to Novikov field $\bb{K}$. There is a natural $A_{\infty}$-embedding
$$
\cu{F}uk(Y,\Lambda_{\Sigma})\hookrightarrow \cu{F}uk_{\bb{K}}^0(Y,\Lambda_{\Sigma}),
$$
where $\cu{F}uk_{\bb{K}}^0(Y,\Lambda_{\Sigma})$ denotes the \emph{tautologically unobstructed immersed Fukaya category over Novikov field $\bb{K}$} whose objects have similar asymptotic behaviour prescribed by $\Lambda_{\Sigma}$. Such asymptotic condition ensures holomorphic disks bounded by such Lagrangian immersion won't escape to infinity and gives rise to a well-controlled Floer theory. It turns out that the category $\cu{F}uk_{\bb{K}}^0(Y,\Lambda_{\Sigma})$ is quasi-equivalent to the exact Fukaya category $\cu{F}uk(Y,\Lambda_{\Sigma})$, but defined over the Novikov field $\bb{K}$ (Proposition \ref{prop:surjective}). In particular, equivariant HMS holds for the Fukaya category $\cu{F}uk_{\bb{K}}^0(Y,\Lambda_{\Sigma})$ and the category of perfect complexes of the toric variety $X_{\Sigma}$ defined over $\bb{K}$. The first question is to determine whether or not 
a given Lagrangian multi-section is an object of $\cu{F}uk_{\bb{K}}^0(Y,\Lambda_{\Sigma})$.
	
	\begin{theorem}[=Theorem \ref{thm:unobs}]
		For $n=2$, any embedded Lagrangian multi-sections of $Y$ are tautologically unobstructed. In particular, we have $HF^{\bullet}(\bb{L},\bb{L})=H_{dR}^{\bullet}(L;\bb{K})$.
	\end{theorem}
	
	Our first main result is to provide a partial symplecto-geometric interpretation of Treumann's result.
	
	\begin{theorem}[=Theorem \ref{thm:degree0}]
		If $\bb{L}\in \cu{F}uk_{\bb{K}}^0(Y,\Lambda_{\Sigma})$ is quasi-isomorphic to a canonically graded Lagrangian multi-section, then for any $m\in M$ and $\sigma\in\Sigma(n)$, the Floer complex $\Hom_{\cu{F}uk(Y)}(D_{m,-\sigma}[-n],\bb{L})$ is concentrated at degree 0, where $D_{m,-\sigma}\subset Y$ is the linking disk associated to $(m,-\sigma)$.
	\end{theorem}
	
	As $\Hom_{\cu{F}uk_{\bb{K}}^0(Y)}(D_{m,-\sigma}[-n],\bb{L})$ computes the microlocal stalk of the sheaf associated to $\bb{L}$ (Theorem \ref{thm:microlocalstalk} in Section \ref{sec:Fuk_microloca} or Theorem 1.1 in \cite{GPS_microlocal}), Theorem 1.4 allows us to deduce the following corollary.
	
	\begin{corollary}[=Corollary \ref{cor:MS_gives_bundle}]
		The mirror of a canonically graded $\Lambda_{\Sigma}$-admissible Lagrangian multi-section is quasi-isomorphic to a toric vector bundle over $X_{\Sigma}$.
	\end{corollary}
	
	In Section \ref{sec:TGRC}, we recall the notion of tropical Lagrangian multi-sections, which was introduced in \cite{branched_cover_fan, Suen_trop_lag}. Each tropical Lagrangian multi-section $\bb{L}^{\trop}$ determines a Lagrangian subset $\Lambda_{\bb{L}^{\trop}}\subset \Lambda_{\Sigma}\subset Y$. We then ask ourselves the following fundamental question.
	
	\begin{question}
		Given a tropical Lagrangian multi-section $\bb{L}^{\trop}$ over a complete fan $\Sigma$, is there a tautologically unobstructed Lagrangian multi-section $\bb{L}\in \cu{F}uk_{\bb{K}}^0(Y,\Lambda_{\Sigma})$ such that $\bb{L}^{\infty}\subset\Lambda_{\bb{L}^{\trop}}^{\infty}$?
	\end{question}
	
	Question 1.6 is called the \emph{Lagrangian realization problem}. If $\bb{L}^{\trop}$ can be realized by a Lagrangian multi-section $\bb{L}$ in the sense of Question 1.6, by Corollary 1.5, its mirror is a toric vector bundle $\cu{E}_{\bb{L}}$ whose associated tropical Lagrangian multi-section $\bb{L}_{\cu{E}_{\bb{L}}}^{\trop}$ is nothing but $\bb{L}^{\trop}$. However, in \cite{Suen_trop_lag}, the second-named author gave an example of a 2-fold tropical Lagrangian multi-section over the fan of $\bb{P}^2$ that \emph{does not} arise from toric vector bundles on $\bb{P}^2$. By mirror symmetry, we should not expect it can be realized by an unobstructed Lagrangian submanifold in $Y$. Hence there must be some extra assumptions that need to be put on $\bb{L}^{\trop}$ in order to get an affirmative answer to the realization problem. In Section \ref{sec:tangent}, we restrict ourselves to 2-fold tropical Lagrangian multi-sections over a complete 2-dimensional fan and introduce the \emph{$N$-generic condition} (Definition \ref{def:N_generic}), which is a pure combinatorial condition on such tropical Lagrangian multi-sections. The reason for us to consider $N$-generic objects is due to the following observation. When the underlying branched covering map of $\bb{L}^{\trop}$ is (topologically) the branched covering map $z\mapsto z^2$ on $\bb{C}$, \cite[Theorem 5.9]{Suen_trop_lag} can be applied to deduce that $\bb{L}^{\trop}$ can be realized by a rank 2 toric vector bundle over $X_{\Sigma}$ if and only if it is $N$-generic with $N\geq 3$ (Proposition \ref{prop:N>2}). Therefore, we may expect the Lagrangian realization problem can be solved when $N\geq 3$. The second main result of this work gives an affirmative answer to this.
	
	\begin{theorem}[=Theorem \ref{thm:unobs_immersed_Lag}]\label{thm:unobs_immersed_Lag-intro}
			Let $\bb{L}^{\trop}$ be a $N$-generic 2-fold tropical Lagrangian multi-section over a complete 2-dimensional fan $\Sigma$ with $N\geq 3$. Then there is a spin, graded and immersed 2-fold Lagrangian multi-section $\bb{L}$ in $Y$, whose immersed sector is concentrated at degree 1 and $\bb{L}^{\infty}\subset\Lambda_{\bb{L}^{\trop}}^{\infty}$. In particular, $\bb{L}$ is tautologically unobstructed. When $\bb{L}$ is embedded, the topology of the underlying surface has Betti numbers $b_0(\bb{L})=1,\,b_1(\bb{L})=N-3,\,b_2(\bb{L})=0$.
	\end{theorem}
	
	\begin{remark}
		A similar realization problem was studied by Hicks in \cite{Hicks_realization}. The tropical object he considered is a tropical subvariety $T$ in a vector space $N_{\bb{R}}$. Contrary to our situation, which is in the setting of 
the SYZ fibration, the Lagrangian submanifold that realizes his tropical subvariety
$T$ does not cover the whole space $N_{\bb{R}}$.  In other words, the mirror sheaf is supported in
 an algebraic subvariety in $(\bb{C}^{\times})^n$, and therefore, cannot be locally free.
	\end{remark}
	
	The construction of $\bb{L}$ in Theorem \ref{thm:unobs_immersed_Lag-intro} is explicit in the sense that the topology of $\bb{L}$ is completely determined. Combining Theorem \ref{thm:unobs_immersed_Lag-intro} with equivariant HMS and Corollary 1.5, this explicit description of the Lagrangian multi-section provides us with the following corollary.
	
	\begin{corollary}[=Corollary \ref{cor:existence_of_bundle}]
		Suppose $\bb{L}^{\trop}$ is a $N$-generic 2-fold tropical Lagrangian multi-section over a complete fan $\Sigma$ on $N_{\bb{R}}\cong\bb{R}^2$ with $N\geq 3$. Then there is an indecomposable rank 2 toric vector bundles $\cu{E}$ on $X_{\Sigma}$ such that
		\begin{align*}
			\dim_{\bb{K}}\Ext_T^0(\cu{E},\cu{E})=&\,1,\\
			\dim_{\bb{K}}\Ext_T^1(\cu{E},\cu{E})=&\,N-3,\\
			\dim_{\bb{K}}\Ext_T^2(\cu{E},\cu{E})=&\,0,
		\end{align*}
		and $\bb{L}_{\cu{E}}^{\trop}=\bb{L}^{\trop}$.
	\end{corollary}

The $N$-genericity condition can actually be generalized to any maximal $r$-fold tropical Lagrangian multi-section over a 2-dimensional complete fan (Definition \ref{def:N_generic_general_r}) and our construction for $r=2$ can be easily generalized and obtain the following partial result for $r\geq 3$.

\begin{theorem}[=Theorem \ref{thm:LRP_higher_rank}]
 A maximal $r$-fold tropical Lagrangian multi-section $\bb{L}^{\trop}$ over a complete 2-dimensional fan $\Sigma$ can be realized by an $r$-fold embedded (and hence unobstructed) Lagrangian multi-section if it is $\left\lfloor 2\left(\frac{d}{r}+1\right)\right\rfloor$-generic for some $d\in\bb{Z}_{>0}$ such that ${\rm{g.c.d.}}(r,d)=1$. In particular, such $\bb{L}^{\trop}$ can be realized by a rank $r$ toric vector bundle over $X_{\Sigma}$.
\end{theorem}

	In Section \ref{sec:rk2}, we apply Theorem 1.7 to prove the following theorem.
	
	\begin{theorem}[=Theorem \ref{thm:general_case}]
		The mirror of a rank 2 indecomposable toric vector bundle $\cu{E}$ on $\bb{P}^2$ is quasi-isomorphic to an embedded, simply connected, and canonically graded Lagrangian multi-section $\bb{L}_{\cu{E}}\subset Y$ \emph{ such that} $\bb{L}_{\cu{E}}^{\infty}\subset\Lambda_{\bb{L}_{\cu{E}}^{\trop}}^{\infty}$.
	\end{theorem}
	
	The proof of Theorem 1.11 relies on the classification result of Kaneyama \cite{Kaneyama_classification}, leading to equivariant rigidity of indecomposable rank 2 toric vector bundles.
	
	\begin{remark}
		The multi-section $\bb{L}_{\cu{E}}$ is exact as it is simply connected. Hence it defines an object in the exact Fukaya category $\cu{F}uk(Y,\Lambda_{\Sigma})$ over $\bb{C}$.
	\end{remark}

	%As a consequence of this theorem, we have
	
	%\begin{corollary}[Corollary \ref{cor:immersed}]
	%Any $\Lambda_{\Sigma_{\bb{P}^2}}$-admissible rank 2 Lagrangian multi-section $\bb{L}$ in $Y$ with non-zero $H^1(L)$ cannot be embedded.
	%\end{corollary}

	\subsection*{Acknowledgment}
	The second author is grateful to Hanwool Bae, Hongtaek Jung, Dogancan Karabas, Man-Chun Lee, Sangjin Lee, Kaoru Ono for useful discussions and suggestions. We also thank Cheol-Hyun Cho, Hansol Hong, Siu-Cheong Lau, Yu-Shen Lin for their interest in this work. Both authors are
	 supported by the IBS project \# R003-D1.

	\section{Preliminaries}\label{sec:prelim}
	
	This section contains a brief review of equivariant coherent-constructible correspondence, microlocalization, and homological mirror symmetry for toric varieties.
	
	\subsection{Equivariant coherent-constructible correspondence}
	
	We denote by $\bb{K}$ the Novikov field over $\bb{C}$. Let $N\cong\bb{Z}^n,N_{\bb{R}}:=N\otimes_{\bb{Z}}\bb{R}$ and $M:=\Hom_{\bb{Z}}(N,\bb{Z}), M_{\bb{R}}:=M\otimes_{\bb{Z}}\bb{R}=N_{\bb{R}}^*$. Throughout the whole paper, $\Sigma$ is a complete fan on $N$ and $X_{\Sigma}$ the associated toric variety over $\bb{K}$. The \emph{equivariant coherent-constructible correspondence} states that we have a quasi-equivalent
	$$\kappa:\cu{P}erf_T(X_{\Sigma})\xrightarrow{\sim} Sh_{cc}(M_{\bb{R}},\Lambda_{\Sigma}),$$
	where $\cu{P}erf_T(X_{\Sigma})$ is the triangulated dg-category of perfect complexes of toric vector bundles on $X_{\Sigma}$ and $Sh_{cc}(M_{\bb{R}},\Lambda_{\Sigma})$ is defined as follows.
	
	We denote by $Sh_{cc}(M_{\bb{R}})$ the triangulated dg-category of complexes of sheaves over $k$ on $M_{\bb{R}}$ whose cohomology sheaves are compactly supported and constructible with respective to some Whitney stratification.

	\begin{definition}[$Sh_{cc}(M_{\bb{R}},\Lambda_{\Sigma})$]
		The category $Sh_{cc}(M_{\bb{R}},\Lambda_{\Sigma})$ is the full subcategory of $Sh_{cc}(M_{\bb{R}})$ whose objects have singular support contained in the conical subset
		\begin{equation}\label{eq:LambdaSigma}
			\Lambda_{\Sigma}:=\bigcup_{\tau\in\Sigma}(\tau^{\perp}+M)\times-\tau
			= \bigcup_{m \in M} \bigcup_{\tau\in\Sigma} (\tau^{\perp}+ m)\times-\tau
		\end{equation}
	\end{definition}
	
	We shall give a brief review on the microlocal stalk and singular support 
	extracted from  \cite{kashiwara-schapira:sheaves}. Let $\cu{F}$ be a complex of sheaves on $M_{\bb{R}}$, its microlocal stalk $\mu_{x,-\xi}(\cu{F})$ at $(x,-\xi)\in T^*M_{\bb{R}}=M_{\bb{R}}\times N_{\bb{R}}$ is defined as follows. By viewing $-\xi\in N_{\bb{R}}$ as a linear function on $M_{\bb{R}}$,  
	we  associate to each $t\in\bb{R}$ and an open subset $U\subset M_{\bb{R}}$ 
	the subcomplex $\Gamma_{-\xi\leq t}(U,\cu{F})$ of $\Gamma(U,\cu{F})$ consisting of sections supported on the closed subset $\{x\in U:-\xi(x)\leq t\}$. 
	
	This assignment $U \mapsto \Gamma_{-\xi\leq t}(U,\cu{F})$ defines a functor from the sheaf $\cu{F}$ to
	the category of abelian groups, which is left exact and so we obtain its right derived functor ${\rm{R}}\Gamma_{-\xi\leq t}(U,-)$. This gives a complex of vector spaces
	$$\mu_{x,-\xi}(\cu{F}):=\lim_{\substack{\longrightarrow\\U\ni x}}{\rm{R}}\Gamma_{-\xi\leq-\xi(x)}(U,\cu{F}),$$
	which is known as the \emph{microlocal stalk of $\cu{F}$ at $(x,-\xi)$}. The \emph{singular support} is defined as
$$
SS(\cu{F}):=\ol{\{(x,-\xi)\in T^*M_{\bb{R}}:\mu_{x,-\xi}(\cu{F})\neq 0\}}.
$$

	When $\cu{F}$ is constructible, $SS(\cu{F})$ is a (singular) Lagrangian subset of the symplectic manifold $(T^*M_{\bb{R}},\omega_{std})$ and when $\cu{F}\in Sh_{cc}(M_{\bb{R}},\Lambda_{\Sigma})$, the stalks $\mu_{x,-\xi}(\cu{F})$ actually glued to a sheaf $\mu_x(\cu{F})$ on $T_x^*M_{\bb{R}}$, which is constructible with respect to $-\Sigma$. In particular, the microlocal stalk is independent of $\xi\in\Int(\sigma)$ and so we simply denote it by $\mu_{x,-\sigma}(\cu{F})$.
	
	Treumann further refines the equivariant CCC by establishing the following theorem.
	
	\begin{theorem}[=Theorem 1.9 in \cite{Morse_theory_TVB}]
		Let $\cu{E}^{\bullet}\in\cu{P}erf_T(X_{\Sigma})$. Then the following statements are equivalent:
		\begin{enumerate}
			\item The complex $\cu{E}^{\bullet}$ is quasi-isomorphic to a 1-term complex concentrated in degree 0.
			\item For each $\sigma\in\Sigma(n)$ and $m\in M$, the microlocal stalk $\mu_{m,-\sigma}(\kappa(\cu{E}^{\bullet}))$ is concentrated at degree 0.
		\end{enumerate}
	\end{theorem}
	
	\subsection{Fukaya categories and microlocalization}\label{sec:Fuk_microloca}

	Let $N,N_{\bb{R}},M,M_{\bb{R}}$ as before. We use $(\xi_j)$ for an affine coordinates of $N_{\bb{R}}$ and $(x_j)$ for those of $M_{\bb{R}}$. Equip $Y:=T^*M_{\bb{R}}=N_{\bb{R}}\times M_{\bb{R}}$ the Liouville 1-form
	$$\lambda:=\sum_{j=1}^n\xi_jdx_j$$
	and the symplectic structure $\omega:=d\lambda$. This gives the volume form
	$$\omega^n:=d\xi_1\wedge dx_1\wedge\cdots\wedge d\xi_n\wedge dx_n.$$
	Equip $Y$ with the complex structure $z_j:=\xi_j+\sqrt{-1}x_j$ and the standard flat metric
	$$
	g = \sum_j |d\xi_j|^2 + |dx_j|^2 = \text{\rm Re} \sum_j d\overline z_j dz_j.
	$$
	Let $\Omega$ be the holomorphic volume form $dz_1\wedge\cdots\wedge dz_n$. Then, we have the compatibility equation
	$$
	\frac{\omega^n}{n!}=(-1)^{\frac{n(n-1)}{2}}\left(\frac{\sqrt{-1}}{2}\right)^n\Omega\wedge\ol{\Omega}.
	$$
	In particular $(T^*M_\bb{R},\omega,\Omega)$ defines a Calabi-Yau manifold structure.
	Denote by $p_{N_{\bb{R}}}:Y\to N_{\bb{R}},\, p_{M_{\bb{R}}}:Y\to M_{\bb{R}}$ the two natural projections. For a point $\xi\in N_{\bb{R}}$, we use the notation $F_{\xi}$ for the fiber $p_{N_{\bb{R}}}^{-1}(\xi)$ of $p_{N_{\bb{R}}}:Y\to N_{\bb{R}}$.
	
	\begin{definition}\label{defn:lag-immersion}
		A Lagrangian immersion of $Y$ is a pair $\bb{L}:=(\til{L},i)$, where $\til{L}$ is a smooth manifold of dimension $n$ and $i:\til{L}\to Y$ is an immersion such that $i^*\omega=0$. We denote the image $i(\til{L})\subset Y$ by $L$. We further assume $L$ has only double point singularities.
	\end{definition}
	
	Recall that a \emph{grading} on an oriented Lagrangian immersion $i:\til{L}\to Y$ is defined to be a choice of smooth function $\theta_{\bb{L}}:\til{L}\to\bb{R}$ so that
	$$\iota^*\Omega=e^{\sqrt{-1}\theta_{\bb{L}}}{\vol}_{\til{L}},$$
	where ${\vol}_{\til{L}}$ is a volume form on $\til{L}$ with $|{\vol}_{\til{L}}|=1$ (with respective to the induced metric).
	
	We introduce a canonical grading on fibers of $p_{N_{\bb{R}}}:Y\to N_{\bb{R}}$ and the zero section $0_{N_{\bb{R}}}$.
	
	\begin{definition}[Canonical grading]\label{def:canonical_orientation}
		For each $\xi\in N_{\bb{R}}$, the \emph{canonical grading} of $F_{\xi}$ is given by $\theta_{F_{\xi}}:=\frac{n\pi}{2}$
		which is induced by the choice of the volume form
		$$
		{\vol}_{\xi}:= dx_1\wedge\cdots\wedge dx_n
		$$
		on the fiber $F_{\xi}\subset Y$. For the zero section $0_{N_{\bb{R}}}$, the \emph{canonical grading} is defined to be $0$
		which is induced by the choice of the volume form
		$$
		{\vol}_{N_{\bb{R}}}:=d\xi_1\wedge\cdots\wedge d\xi_n.
		$$
	\end{definition}

	Let $\bb{L}_1,\bb{L}_2$ be two graded Lagrangian immersions so that their images $L_1,L_2$ intersect transversally away from their self-intersection points. The \emph{degree} of an intersection point $p\in L_1\cap L_2$ is defined as follows. As $Y\cong\bb{C}^n$ as Calabi-Yau manifolds, there is a $U(n)$ matrix $\psi_p:\bb{C}^n\to \bb{C}^n$ so that
	\begin{align*}
		\psi_p(T_pL_1)=&\,\{(x_1,\dots,x_n)\in\bb{C}^n:x_i\in\bb{R}\},\\
		\psi_p(T_pL_2)=&\,\{(e^{\sqrt{-1}\theta_1}x_1,\dots,e^{\sqrt{-1}\theta_n}x_n)\in\bb{C}^n:x_i\in\bb{R}\},
	\end{align*}
	for some $\theta_i\in(0,\pi)$. We call these angles the \emph{angles of intersection at $p$}. The degree of $p$ from $\bb{L}_1$ to $\bb{L}_2$ is defined to be
	$$
	\deg_{\bb{L}_1,\bb{L}_2}(p):=\frac{1}{\pi}(\theta_1+\cdots+\theta_n+\theta_{\bb{L}_1}(p)-\theta_{\bb{L}_2}(p))\in\bb{Z}.
	$$
	We have $\deg_{\bb{L}_1,\bb{L}_2}(p)+\deg_{\bb{L}_2,\bb{L}_1}(p)=n$.
	
	The Fukaya category of $Y$ we consider here are those considered in \cite{CCC_HMS}. We recall their definitions. Let $(X,g)$ be a Riemannian manifold. Consider the unit disk bundle
	$$D^*X:=\{(x,\xi)\in T^*X:\|\xi\|_g\leq 1\}\subset T^*X\}$$
	and the unit sphere bundle
	$$S^*X:=\{(x,\xi)\in T^*X:\|\xi\|_g=1\}\subset D^*X.$$
	There is an embedding $\iota:T^*X\to D^*X$, defined by
	$$\iota:(x,\xi)\mapsto\left(x,\frac{\xi}{\sqrt{1+\|\xi\|_g^2}}\right).$$
	Given any subset $L\subset T^*X$, we define
	$$L^{\infty}:=\ol{\iota(L)}\cap D^*X.$$
	
	\begin{definition}\label{def:fuk}
		Fix an analytic-geometric category $\cu{C}$. Let $X$ be a real analytic manifold and equip $T^*X$ with the standard Liouville structure $\lambda$. A \emph{Lagrangian brane} of $T^*X$ is a pair $(\bb{L},E)$, where $\bb{L}$ is a graded $\lambda$-exact Lagrangian immersion with at most finitely many double point singularities and spin structure on the domain $\til{L}$ so that
		\begin{enumerate}
			\item $p_X(L)\subset X$ is compact,
			\item $\ol{\iota(L)}$ is a $\cu{C}$-set (see Appendix A.1. of \cite{CCC_HMS}),
			\item $L$ admits a tamed perturbation data as in \cite{NZ},
		\end{enumerate}
		and $E$ is a $U_{\bb{K}}$-local system on $L$, where $U_{\bb{K}}:=\ker({\rm{val}}:\bb{K}^{\times}\to\bb{R})$.
	\end{definition}
	
	Let $(\bb{L}_1,E_1),(\bb{L}_2,E_2)$ be immersed Lagrangian branes of $Y$. To define their morphism space, Nadler-Zaslow introduced the notion of a fringed set.
	
	\begin{definition}
		An open subset $R_2\subset\bb{R}_{>0}^2$ is called fringed if $\pi_2:R_2\to\bb{R}$ has image being an interval $(0,\delta)$ and whenever $(\delta_1,\delta_2)\in R$, we have $(\delta_1',\delta_2)\in R$ for all $\delta_1'\in(0,\delta_1)$.
	\end{definition}
	
	They also proved that for given controlled Hamiltonian functions\footnote{A controlled Hamiltonian function is a function $H:Y\to\bb{R}$ so that outside a vertically compact set $K$, containing the zero section, we have $H(x,\xi)=|\xi|$ with respective the the flat metric.} $H_1,H_2$, there exists a fringed set $R_2$ so that for any $(\delta_1,\delta_2)\in R_2$, there exists $r>0$ so that
	$$
	\ol{\varphi_{H_1,\delta_1}(L_1)}\cap \ol{\varphi_{H_2,\delta_2}(L_2)}\subset Y_{\leq r}: = Y \cap \{\|\xi\|_g \leq r\},
	$$
	The morphism space from $\bb{L}_1$ to $\bb{L}_2$ is then defined to be
	$$CF(\bb{L}_1,\bb{L}_2):=\bigoplus_{p\in\varphi_{H_1,\delta_1}(L_1)\cap\varphi_{H_2,\delta_2}(L_2)}\Hom(E_{1,p},E_{2,p}),$$
	which is $\bb{Z}$-graded. The Floer differential and the $A_{\infty}$-infinity maps are defined as usual.
	
	\begin{definition}
		The triangulated envelope of the $A_{\infty}$-category generated by embedded exact Lagrangian branes of $T^*X$ is denoted by $\cu{F}uk(T^*X)$. Let $\Lambda\subset T^*X$ be a conical Lagrangian subset. A Lagrangian brane $(L,E)$ is said to be \emph{$\Lambda$-admissible} if it further satisfies the asymptotic condition $L^{\infty}\subset\Lambda^{\infty}$. The $A_{\infty}$-category generated by $\Lambda$-admissible embedded Lagrangian branes is denoted by $\cu{F}uk(T^*X,\Lambda)$, which is a full subcategory of $\cu{F}uk(T^*X)$.
	\end{definition}
	
	Now, we apply Definition \ref{def:fuk} to $X=M_{\bb{R}}$ and the conical Lagrangian
	$$
	\Lambda_{\Sigma}:=\bigcup_{\tau\in\Sigma}(\tau^{\perp}+M)\times-\tau
	$$
	This gives two Fukaya categories $\cu{F}uk(Y),\cu{F}uk(Y,\Lambda_{\Sigma})$.
	
	In \cite{NZ, Nadler}, Nadler-Zaslow proved that there is a quasi-equivalence
	$$\mu:Sh_{cc}(M_{\bb{R}})\xrightarrow{\simeq} \cu{F}uk(Y)$$
	which restricts to a quasi-equivalence
\begin{equation}\label{eq:muSigma}
	\mu_{\Sigma}:Sh_{cc}(M_{\bb{R}},\Lambda_{\Sigma})\xrightarrow{\simeq} \cu{F}uk(Y,\Lambda_{\Sigma}).
\end{equation}
	Their construction goes as follows. First, they identify $Sh_{cc}(M_{\bb{R}})$ as the triangulated envelope of the Morse category $Mor(M_{\bb{R}})$ (so they share the same derived category), whose objects are given by the pair $(U,f)$ where $U\subset M_{\bb{R}}$ is a bounded open subset and $f$ is defined as follows: 
	We take a nonnegative real-analytic function  $m:\ol{U}\to\bb{R}_{\geq 0}$ such that
	 $m(x)=0$ if and only if $x \in \partial U$. Then we write $f : = \log m$ which is a 
	function defined on $U$ that blows up along the boundary $\partial \overline U$.
	
	The morphism space is given by the relative de Rham complex
	$$
	\Hom_{Sh_{cc}(M_{\bb{R}})}((U_0,f_0),(U_1,f_1))
	:=(\Omega_{dR}(\ol{U}_0\cap U_1,\partial U_0\cap U_1),d).
	$$
	The functor $Mor(M_{\bb{R}})\to Sh_{cc}(M_{\bb{R}})$ is given by
	$$(U,f)\mapsto i_{U*}\ul{\bb{K}}_U,$$
	where $i_U:U\to M_{\bb{R}}$ is the inclusion and $\ul{\bb{K}}_U$ denotes the $\bb{K}$-valued constant sheaf on $U$. 
We can define $Mor(M_{\bb{R}})\to \cu{F}uk(Y)$ by mapping $(U,f)$ to the Lagrangian
	$$
	L_{U*}:=\Gamma_{d\log(f)}:=\{(x,d\log(f(x)))\in Y:x\in U\}.
	$$
	(See \cite{kasturirangan}, \cite{kasturirangan-oh} for an earlier appearance of such  correspondence in the Floer theory.) This Lagrangian carries a natural brane structure and is called a \emph{standard brane}. Then $\mu$ is defined by taking $i_{U*}\ul{\bb{K}}_U$ to $L_{U*}$. Taking Verdier dual, we have the costandard object $i_{U!}\ul{\bb{K}}_U[n]$, which gets mapped to the \emph{costandard brane}
	$$L_{U!}:=\Gamma_{-d\log(f)}.$$
	This defines $\mu:Sh_{cc}(M_{\bb{R}})\to \cu{F}uk(Y)$ and its restriction \eqref{eq:muSigma}.
	
	\section{Microlocalization functor over the Novikov field}
	\label{sec:microlocalization-over-Novikov}
	
	For the purpose of this paper, exact Lagrangian submanifolds are not enough. We also need to consider some non-exact Lagrangian immersions. Therefore, we need to extend our coefficient field to the Novikov field. Moreover, as our Lagrangians are non-compact, we need to make sure holomorphic disks won't escape to infinity to ensure compactness of their moduli spaces. (We omit discussion on this $C^0$-estimate and refer readers to \cite{NZ} for the details. See Remark \ref{rem:choiceJ} for some comments related to the choice of $J$ in this regard.)
	
	We now follow \cite{AJ} to introduce the moduli space of holomorphic disks bounded by $L$. Let $R$ denote the set of points $(p^+,p^-)\in L\times_YL$ for which $p^+\neq p^-$. Let $I\subset\{0,1,\dots,k\}$ and $\alpha:I\to R$. Denote by $\Sigma$ a genus 0 pre-stable bordered Riemann surface and $z=(z_0,z_1,\dots,z_k)$ are distinct smooth points on $\partial\Sigma$. There is a continuous map $l:S^1\to\partial\Sigma$ so that $l^{-1}(z)$ consists of two points if $z\in\partial\Sigma$ is a singular point, otherwise, $l^{-1}(z)=\{\zeta\}$ has only one point. Let $\ol{\cu{M}}_{k+1}(Y,\bb{L};\alpha,\beta)$ be the Kuranishi space, which consists of the data $(\Sigma,z,u,l,\ol{u})$ modulo obvious automorphisms, satisfying the following properties:
	\begin{itemize}
		\item [(a)] $u:\Sigma\to Y$ is a holomorphic map so that $u(\partial\Sigma)\subset L$ and $u_*([\Sigma])=\beta\in\pi_2(Y,L)$.
		\item [(b)] $\zeta_0,\zeta_1,\dots,\zeta_k\in S^1$ are ordered counter-clockwise.
		\item [(c)] $\ol{u}:S^1\backslash\{\zeta_i\}_{i\in I}\to L$ is a continuous map so that $u\circ l=i\circ\ol{u}$.
		\item [(d)] For any $i\in I$, we have
		$$\alpha(i)=\left(\lim_{\theta\to 0^-}\ol{u}(e^{\sqrt{-1}\theta}\zeta_i),\lim_{\theta\to 0^+}\ol{u}(e^{\sqrt{-1}\theta}\zeta_i)\right)\in R.$$
	\end{itemize}
	Define $\ev_i:\ol{\cu{M}}_{k+1}(Y,\bb{L};\alpha,\beta)\to L\sqcup R$ to be the $i$-th evaluation map, given by
	$$
	\ev_i([\Sigma,z,u,l,\ol{u}]):=\begin{cases}
	\ol{u}(\zeta_i)\in L & \text{ if } i\notin I,\\
	\alpha(i)\in R & \text{ if }i\in I.
	\end{cases}
	$$
	Let $\sigma:R\to R$ be the obvious involution. We further define $\ev:\ol{\cu{M}}_{k+1}(Y,\bb{L};\alpha,\beta)\to L\sqcup R$ to be
	$$
	\ev([\Sigma,z,u,l,\ol{u}]):=\begin{cases}
	\ol{u}(\zeta_0)\in L & \text{ if } 0\notin I,\\
	\sigma(\alpha(0))\in R & \text{ if }0\in I.
	\end{cases}
	$$
	We put
	$$\ol{\cu{M}}_{k+1}(Y,\bb{L};\beta):=\bigsqcup_{\substack{I\subset\{0,1,\dots,k\}\\\alpha:I\to R}}\ol{\cu{M}}_{k+1}(Y,\bb{L};\alpha,\beta),$$
	which is a union of Kuranishi spaces with possibly different virtual dimensions. Denote by
	$$
	{\bf{ev}},{\bf{ev}}_i:\ol{\cu{M}}_{k+1}(Y,\bb{L};\beta)\to L\sqcup R
	$$
	the obvious evaluation maps. For later use, we also let $\ol{\cu{M}}_{k+1,l}(Y,\bb{L};\beta)$ denote the moduli space of holomorphic disks bounded $L$ with $k+1$ boundary marked points and $l$ interior marked points. The tameness condition (3) in Definition \ref{def:fuk} ensures compactness of $\ol{\cu{M}}_{k+1,l}(Y,\bb{L};\beta)$.
	
	To define the $A_{\infty}$-maps, we use the de Rham model
	$$
	CF^{\bullet}(\bb{L};\bb{K}):=H_{dR}^{\bullet}(\til{L};\bb{K})\oplus\bb{K}\inner{R}\cong H^{\bullet}(\til{L}\times_Y\til{L}).
	$$
	Each element $(p^-,p^+)\in R$ can be absolutely graded by the Fredholm index of certain Cauchy-Riemann operator (See \cite{AJ}, Section 4.3). We define $m_k:CF^{\bullet}(\bb{L};\bb{K})^{\otimes k}\to CF^{\bullet}(\bb{L};\bb{K})$ by
	$$
	m_k(\alpha_1,\dots,\alpha_k):=\sum_{\beta\in\pi_2(Y,L)}{\bf{ev}}_!({\bf{ev}}_1^*(\alpha_1)\cup\cdots\cup{\bf{ev}}_k^*(\alpha_k)) T^{\omega(\beta)}.
	$$
	Then it is a standard argument that $(CF^{\bullet}(\bb{L};\bb{K}),\{m_k\}_{k\geq 0})$ defines an $A_{\infty}$-algebra.

	Recall a bounding cochain is a degree 1 element  $b\in CF^1(L,\bb{K}^+)$ such that
	$$\sum_{k=0}^{\infty}m_k(b^{\otimes k})=0.$$
	Given a bounding cochain $b$, one can deform the $A_{\infty}$-algebra $m_k$ by
	$$m_k^b(\alpha_1,\dots,\alpha_k):=\sum_{l_1,\dots,l_k\geq 0}m_{k+l_1+\cdots+l_k}(b^{\otimes l_1},\alpha_1,b^{\otimes l_2},\alpha_2,\dots,b^{\otimes l_k},\alpha_k).$$
	This gives a new $A_{\infty}$-algebra with $(m_1^b)^2=0$. Hence the Floer cohomology for $(\bb{L},b)$ is well-defined.
	
	\begin{definition}
		An immersed Lagrangian brane $\bb{L}$ is said to be \emph{unobstructed} if it admits a bounding cochain. We said $\bb{L}$ is \emph{tautologically unobstructed} if $0$ is a bounding cochain.
	\end{definition}
	
	For a pair of unobstructed immersed Lagrangian branes $(\bb{L}_1,b_1),(\bb{L}_2,b_2)$, one can also define their Floer cohomology by setting
	$$\mf{n}^{b_1,b_2}(p):=\sum_{l_1,l_2\geq 0}m_{l_1+1+l_2}^{\bb{L}_1,\bb{L}_2}(b_1^{\otimes l_1},p,b_2^{\otimes l_2}),$$
	where $m_{l_1+1+l_2}^{\bb{L}_1,\bb{L}_2}$ is defined by counting holomorphic disks bounded by $L_1,L_2$ with $l_1+l_2+2$ marked points so that the first $l_1$ marked points land on $b_1$, the $(l_1+1)$-th marked point mapped to the input $p\in i_1(L_2)\cap L_2$, the $(l_1+2),\dots,(l_1+l_2+1)$-th marked points go to $b_2$ and the last one goes to an output in $i_1(L_2)\cap L_2$. It follows from the $A_{\infty}$-relation that $(m_1^{b_1,b_2})^2=0$.
	
	In this paper, we are interested in tautologically unobstructed Lagrangian immersions.
	\begin{definition}\label{defn:FukK}
		Denote by $\cu{F}uk_{\bb{K}}^0(Y)$ the $A_{\infty}$-category generated by tautologically unobstructed Lagrangian immersions equipped with $\bb{K}$-valued local system and by $\cu{F}uk_{\bb{K}}^0(Y,\Lambda_{\Sigma})$ the full subcategory generated by $\Lambda_{\Sigma}$-admissible ones.
	\end{definition}
	The derived category of compactly support constructible sheaves on $M_{\bb{R}}$ over $\bb{K}$ is denoted by $Sh_{cc,\bb{K}}(M_{\bb{R}})$ and by $Sh_{cc,\bb{K}}(M_{\bb{R}},\Lambda_{\Sigma})$ the full subcategory generated by those sheaves with microlocal support in $\Lambda_{\Sigma}$.
	There is a microlocalization functor
	$$\mu:Sh_{cc,\bb{K}}(M_{\bb{R}})\to \cu{F}uk_{\bb{K}}^0(Y)$$
	given by mapping $i_{U*}\ul{\bb{K}}_{U}$ to $L_{U*}$ equipped with the trivial $U_{\bb{K}}$-local system: Thanks to the
	sequence of isomorphisms
	\begin{align*}
		\Hom_{Sh_{cc,\bb{K}}(M_{\bb{R}})}(i_{U_1*}\ul{\bb{K}}_{U_1},i_{U_2*}\ul{\bb{K}}_{U_2})\simeq & \,(\Omega_{dR}^{\bullet}(\ol{U}_1\cap U_2,\partial U_1\cap U_2;\bb{K}),d)\\
		\simeq & \,(\Omega_{dR}^{\bullet}(\ol{U}_1\cap U_2,\partial U_1\cap U_2;\bb{C})\otimes_{\bb{C}}\bb{K},d)\\
		\simeq & \,\Hom_{\cu{F}uk(Y)}(L_{U_1*},L_{U_2*})\otimes_{\bb{C}}\bb{K}\\
		\simeq & \,\Hom_{\cu{F}uk_{\bb{K}}^0(Y)}(L_{U_1*},L_{U_2*}),
	\end{align*}
	the functor $\mu$ is well-defined and fully faithful.
	
	\begin{proposition}\label{prop:surjective}
		The microlocalization functor $\mu:Sh_{cc,\bb{K}}(M_{\bb{R}})\to \cu{F}uk_{\bb{K}}^0(Y)$ is a quasi-equivalence. Hence $\cu{F}uk_{\bb{K}}^0(Y)$ is quasi-equivalent to the exact Fukaya category $\cu{F}uk_{\bb{K}}(Y)$ over $\bb{K}$.
	\end{proposition}
	\begin{proof}
		By going through the work of Nadler \cite{Nadler}, we find that for the tautological unobstructed category, the microlocalization functor is still essentially surjective.
		In \cite{Nadler}, Nadler turned the $\cu{F}uk(Y)$-module $\Hom_{\cu{F}uk(Y)}(P,L)$ into a twisted complex of standard modules
		$$
		\Hom_{\cu{F}uk(Y)}(\alpha_{M_{\bb{R}}}(L),L_{\{x_a\}})\otimes \Hom_{\cu{F}uk(Y)}(P,L_{\tau_a*}),
		$$
		where $\cu{T}:=\{\tau_a\}$ gives a fine enough stratification of $M_{\bb{R}}$ so that $L^{\infty}\subset\Lambda_{\cu{T}}^{\infty}$ and $\alpha_{M_{\bb{R}}}:T^*M_{\bb{R}}\to T^*M_{\bb{R}}$ is the involution $(x,\xi)\mapsto(x,-\xi)$.
		(Nadler's proof relies on the result that the diagonal of $T^*M_\bR^- \times T^*M_\bR$ can be
		resolved by the standard branes via the argument of utilizing the Lagrangian correspondence functor, which holds the case
		independent of the coefficient field.)
		The only difference here is that the (finite-dimensional) vector spaces $\Hom_{\cu{F}uk(Y)}(\alpha_{M_{\bb{R}}}(L),L_{\{x_a\}}),\Hom_{\cu{F}uk(Y)}(P,L_{\tau_a*})$ and morphisms are now defined over $\bb{K}$ when $L,P$ are tautologically unobstructed. For readers' convenience, we provide 
		some more explanation in  Appendix \ref{appendix:nadler}
		on how Nadler's proof given in \cite[Section 4]{Nadler} applies to
		our case of tautologically unobstructed immersed Lagrangians.
	\end{proof}
	
	As CCC \cite{CCC_HMS}
	is true over any field, we obtain the equivariant homological mirror symmetry over the Novikov field.
	
	\begin{corollary}\label{cor:eqivariant_HMS}
		Let $\Sigma$ be a complete fan on $N_{\bb{R}}$ and $X_{\Sigma}$ be its associated toric variety over $\bb{K}$. There is a quasi-equivalence $\cu{P}erf_T(X_{\Sigma})\simeq \cu{F}uk_{\bb{K}}^0(Y,\Lambda_{\Sigma})$.
	\end{corollary}

	Let $(\bb{L},\cu{L})\in \cu{F}uk_{\bb{K}}^0(Y,\Lambda_{\Sigma})$ and $\cu{F}_{(\bb{L},\cu{L})}\in Sh_{cc,\bb{K}}(M_{\bb{R}},\Lambda_{\Sigma})$ be the corresponding constructible sheaf over $\bb{K}$. We want to compute the microlocal stalk of $\cu{F}_{(\bb{L},\cu{L})}$.
	Recall that $\Lambda_{\Sigma}^{\infty}$ is a (singular) Legendrian of $S^*M_{\bb{R}}$.
	We decompose it into
	$$
	\Lambda_{\Sigma}^{\infty}:=\Lambda_{sm}^{\infty}\sqcup\Lambda_{sing}^{\infty},
	$$
	the smooth and singular part. Note that for each $m\in M$ and $\sigma\in\Sigma(n)$, the subset $\{m\}\times\Int(-\sigma)^{\infty}$ is a connected component of $\Lambda_{sm}^{\infty}$.
	For each point $(m,-\xi^{\infty})\in\{m\}\times\Int(-\sigma)^{\infty}$, there is a small Legendrian linking sphere $S_{(m,-\xi^{\infty})}\subset S^*M_{\bb{R}}$ around the point $(m,-\xi^{\infty})$.
	
	Propagating $S_{(m,-\xi^{\infty})}$  into $D^*M_\bR$
	along a path in $\Int(\sigma)$ so that the radius of $S_{(m,-\xi^{\infty})}$ shrink to 0, we obtain a Lagrangian disk which we denote by $D_{(m,-\xi^{\infty})}$. It was argued in \cite[Section 3.3]{GPS_sectorial_descent} that the quasi-isomorphism class of the Lagrangian disk $D_{(m,-\xi^{\infty})}$ only depends on the component $\{m\}\times\Int(-\sigma)^{\infty}$ as long as $D_{(m,-\xi^{\infty})}$ doesn't hit the zero section $0_{M_{\bb{R}}}$ of $p_{M_{\bb{R}}}:Y\to M_{\bb{R}}$. Hence we can simply write $D_{m,-\sigma}$ for $D_{(m,-\xi^{\infty})}$. In particular, for any $\xi\in\Int(\sigma)$, we can move $D_{m,-\sigma}$ by a Hamiltonian flow that is parallel to the Liouville flow so that $D_{m,-\sigma}\cap F_{-\xi}$ is an open ball in $F_{-\xi}$. We denote such a choice of representative by $D_{m,-\xi}$. As $D_{m,-\xi}$ is contractible, it must be exact and graded. We equip $D_{m,-\xi}$ the unique grading $\theta_{m,-\xi}$ so that $\theta_{m,-\xi}=\frac{n\pi}{2}$ on $D_{m,-\xi}\cap F_{-\xi}$.
	
	\begin{definition}
		For $m\in M$, $\sigma\in\Sigma(n)$, the Lagrangian disk $D_{m,-\sigma}$ is called a \emph{linking disk}.
	\end{definition}
	
	For each $m\in M$, we choose a small neighbourhood $U_m\subset M_{\bb{R}}$. For any $\xi\in N_{\bb{R}}$, we put
	$$U_{-\xi>-\xi(m)}:=U_m\cap\{x\in M_{\bb{R}}:-\xi(x)>-\xi(m)\}.$$
	Then we obtain two costandard objects $L_{U_m!},L_{U_{-\xi>-\xi(m)}!}\in \cu{F}uk(Y)$. Their morphism space is given by
	\begin{align*}
		\Hom_{\cu{F}uk(Y)}(L_{U_{-\xi>-\xi(m)}!},L_{U_m!})\simeq &\,  \Hom_{Sh_{cc}(M_{\bb{R}})}(i_{U_{-\xi>-\xi(m)}!}\ul{\bb{K}}_{U_{-\xi>-\xi(m)}},i_{U!}\ul{\bb{K}}_{U_m})\\
		\simeq & \,(\Omega_{dR}(U_{-\xi>-\xi(m)})\otimes_{\bb{C}}\bb{K},d).
	\end{align*}
	Hence their Floer cohomology is generated by a degree 0 element $p_0$.
	It follows from a fundamental result in \cite[Theorem 1.9]{GPS_sectorial_descent} that one has the exact triangle\footnote{According to \cite{GPS_sectorial_descent}, the first two terms in (\ref{eqn:exact_triangle}) seems need to be reversed. This difference is caused by the fact that the authors in \cite{GPS_sectorial_descent} actually perturb $L_1$ more than $L_2$ to define the (wrapped) Floer cohomology while in here or \cite{NZ, Nadler}, $L_2$ is perturbed more. This explains the reversal of the first two terms in (\ref{eqn:exact_triangle}). Also note that in \cite{GPS_microlocal}, Section 4.4, there is a shift of degree when defining the microlocal stalk.}
	\begin{equation}\label{eqn:exact_triangle}
		L_{U_{-\xi>-\xi(m)}!}\xrightarrow{p_0} L_{U_m!}\to D_{m,-\xi},
	\end{equation}
	The following theorem should follow from \cite[Theorem 1.1]{GPS_microlocal} by restricting to the infinitesimal wrapped subcategory. But for self-containedness, we give a direct proof for it without using the wrapped result.
	
	\begin{theorem}\label{thm:microlocalstalk}
		Let $\bb{L}\in \cu{F}uk_{\bb{K}}^0(Y,\Lambda_{\Sigma})$ and $\cu{F}_{\bb{L}}\in Sh_{cc,\bb{K}}(M_{\bb{R}},\Lambda_{\Sigma})$ be the corresponding constructible sheaf over $\bb{K}$. Then
		$$\mu_{m,-\sigma}(\cu{F}_{\bb{L}})\simeq \Hom_{\cu{F}uk_{\bb{K}}^0(Y)}(D_{m,-\sigma}[-n],\bb{L}),$$
		for all $m\in M$ and $\sigma\in\Sigma(n)$.
	\end{theorem}
	\begin{proof}
		By definition of the sheaf $\cu{F}_{\bb{L}}$, we have
		$$
		\Hom_{Sh_{cc,\bb{K}}(M_{\bb{R}})}(i_{U!}\ul{\bb{K}}_U,\cu{F}_{\bb{L}})\simeq \Hom_{\cu{F}uk_{\bb{K}}^0(Y)}(L_{U!}[-n],\bb{L}).
		$$
		In particular, for $\xi\in\Int(\sigma)$, we have
		$$
		\Hom_{Sh_{cc,\bb{K}}(M_{\bb{R}})}(i_{U_{-\xi>-\xi(m)}!}\ul{\bb{K}}_{U_{-\xi>-\xi(m)}},\cu{F}_{\bb{L}})\simeq \Hom_{\cu{F}uk_{\bb{K}}^0(Y)}(L_{U_{-\xi>-\xi(m)}!}[-n],\bb{L})
		$$
		and the restriction map
		$$
		\Hom_{Sh_{cc,\bb{K}}(M_{\bb{R}})}(i_{U!}\ul{\bb{K}}_{U_m},\cu{F}_{\bb{L}})\to \Hom_{Sh_{cc,\bb{K}}(M_{\bb{R}})}(i_{U_{-\xi>-\xi(m)}!}\ul{\bb{K}}_{U_m},\cu{F}_{\bb{L}})
		$$
		can be identified with the multiplication map
		$$
		m_2(-,p_0):\Hom_{\cu{F}uk_{\bb{K}}^0(Y)}(L_{U_m!}[-n],\bb{L})\to \Hom_{\cu{F}uk_{\bb{K}}^0(Y)}(L_{U_{-\xi>-\xi(m)}!}[-n],\bb{L}),
		$$
		(See \cite[Lemma 3.26]{ganatra2020covariantly} for this identification.) Let $\cu{F}_{-\xi\leq-\xi(m)}$ be the kernel of the restriction map $\cu{F}_{\bb{L}}\to\cu{F}_{\bb{L}}|_{U_{-\xi>-\xi(m)}}$. Then we obtain a short exact sequence of compelxes of sheaves
		$$0\to\cu{F}_{-\xi\leq-\xi(m)}\to\cu{F}_{\bb{L}}\to\cu{F}_{\bb{L}}|_{U_{-\xi>-\xi(m)}}\to 0,$$
		which induces the long exact sequence in hypercohomologies
		$$\bb{H}^{\bullet}(U_m,\cu{F}_{-\xi\leq-\xi(m)})\to\bb{H}^{\bullet}(U_m,\cu{F}_{\bb{L}})\to\bb{H}^{\bullet}(U_m,\cu{F}_{\bb{L}}|_{U_{-\xi>-\xi(m)}})\xrightarrow{[1]}.$$
		The cohomology of the microlocal stalk $\mu_{m,-\xi}(\cu{F}_{\bb{L}})$ is by definition
		$$
		H^{\bullet}(\mu_{m,-\xi}(\cu{F}_{\bb{L}})):=\lim_{\substack{\longrightarrow\\U\ni m}}\bb{H}^{\bullet}(U,\cu{F}_{-\xi\leq-\xi(m)})
		$$
		and we also have the following commutative diagram
		\begin{equation*}
			\xymatrix{{} \ar[r] & \bb{H}^{\bullet}(U,\cu{F}_{-\xi\leq-\xi(m)})  \ar[r] \ar[d]
				& {\bb{H}^{\bullet}(U_m,\cu{F}_{\bb{L}})} \ar[d]^{\cong} \ar[r]^(0.4){Res}
				& {\bb{H}^{\bullet}(U_m,\cu{F}_{\bb{L}}|_{U_{-\xi>-\xi(m)}})} \ar[d]^{\cong} \ar[r] &{}
				\\{}  \ar[r] & HF^{\bullet}(D_{m,-\sigma}[-n],\bL) \ar [r]
				& HF^{\bullet}(L_{U_m!}[-n],\bb{L}) \ar[r]^{m_2(-,p_0)\quad\,}
				& HF^{\bullet}(L_{U_{-\xi>-\xi(m)}!}[-n],\bb{L}) \ar[r]&  {}
			}
		\end{equation*}
		The exact triangle (\ref{eqn:exact_triangle}) induces the long exact sequence of Floer cohomologies
		$$
		HF^{\bullet}(D_{m,-\sigma}[-n],\bb{L})\to HF^{\bullet}(L_{U_m!}[-n],\bb{L})\to HF^{\bullet}(L_{U_{-\xi>-\xi(m)}!}[-n],\bb{L})\xrightarrow{[1]},
		$$
		which leads us to the conclusion by the Five Lemma that
		$$H^{\bullet}(\mu_{m,-\sigma}(\cu{F}_{\bb{L}}))\cong HF^{\bullet}(D_{m,-\sigma}[-n],\bb{L})$$
		as desired.
	\end{proof}
	
	\section{Lagrangian multi-sections and their equivariant mirror}\label{sec:LMS}
	
	We now introduce the notion of Lagrangian multi-sections for the projection map $p_{N_{\bb{R}}}:Y\to N_{\bb{R}}$. First, we recall the notion of branched covering maps between smooth manifolds of the same dimension.
	
	\begin{definition}\label{def:branced_covering}
		Let $M_1,M_2$ be smooth manifolds with the same dimension. A smooth map $f:M_1\to M_2$ is called a \emph{$r$-fold branched covering map} if there exists a codimension 2 subset $B_f\subset M_2$ such that the restriction $f|_{M_1\backslash f^{-1}(B_f)}:M_1\backslash f^{-1}(B_f)\to M_2\backslash B_f$ is an $r$-fold covering map. The set $B_f$ is called the \emph{branch locus} of $f$ and the set of all points at which $f$ fails to be a local diffeomorphism is called the \emph{ramification locus} of $f$. We also put $S_f:=f^{-1}(B_f)$.
	\end{definition}
	
	\begin{remark}
		It is known that the ramification locus of a branched covering map is a union of locally closed submanifolds of codimension at least 2.
	\end{remark}
	
	\begin{definition}\label{def:LMS}
		An immersed Lagrangian $\bb{L}:=(i:\til{L}\to Y)$ is called a \emph{Lagrangian multi-section of degree $r$} if the composition $p_{\bb{L}}:=p_{N_{\bb{R}}}\circ i:\til{L}\to N_{\bb{R}}$ is a branched $r$-fold covering map. We simply write $B_{\bb{L}}$ for $B_{p_{\bb{L}}}$ and $S_{\bb{L}}$ for $S_{p_{\bb{L}}}$. We also put
		$$I_{\bb{L}}:=\{p\in L:p\text{ is an immersed point}\},$$
		which is assumed to be finite. The \emph{canonical orientation} of $\bb{L}$ is the orientation on $\bb{L}$ so that $p_{\bb{L}}:\til{L}\to N_{\bb{R}}$ is orientation preserving.
	\end{definition}
	
	While single-valued Lagrangian sections are exact and so (tautologically) unobstructed in our
	circumstance, Lagrangian multi-sections are not necessarily exact when the branch locus is non-empty.
	
	\subsection{Tautological unobstructedness of Lagrangian multi-sections in dimension 2}
	
	In this subsection,  we examine tautological unobstructedness of general
	Lagrangian multi-sections of dimension 2. The following is a result in this direction directly relevant to
	our main purpose.
	
	\begin{theorem}\label{thm:unobs}
		Any embedded 2-dimensional Lagrangian multi-sections of $Y$ are tautologically unobstructed. In particular, we have $HF^{\bullet}(\bb{L},\bb{L})=H_{dR}^{\bullet}(L;\bb{K})$.
	\end{theorem}
	
	By definition, tautological unobstructedness will follow if we prove
	that $\bb{L}$ bounds no rigid holomorphic disks for a generic perturbation of
	$\bL$ or almost complex structures $J$.

	We start with the following proposition.
	\begin{lemma}\label{lem:intersection} Let $\bL$ be any Lagrangian multi-section.
		Any non-constant holomorphic disks bounded by $L$ must intersect the set
		$p_{N_{\bb{R}}}^{-1}(B_{\bb{L}})$.
	\end{lemma}
	\begin{proof} 
		Let $Y_0:=Y\backslash p_{N_{\bb{R}}}^{-1}(B_{\bb{L}})$. We first note that
		 $p_{\bb{L}}$ restricts to an unbranched covering map 
		 $L_0:=L\backslash S_{\bb{L}}\to N_{\bb{R}}\backslash B_{\bb{L}}$. 
		 We will then prove the lemma by contradiction.		 
		 Suppose $u:(D^2,\partial D^2)\to(Y_0,L_0)$ is a holomorphic disk
		that does not intersect $p_{N_{\bb{R}}}^{-1}(B_{\bb{L}})$. Consider the composition 
		$w:=p_{N_{\bb{R}}} \circ u:D^2\to N_{\bb{R}}$. By the standing assumption 
		$w(D^2) \cap B_{\bb{L}} = \emptyset$. Since 
		$L_0:=L\backslash S_{\bb{L}}\to N_{\bb{R}}\backslash B_{\bb{L}}$ is an unbranched covering,
		it admits a lifting $w':D^2 \to L_0$ so that $[u\#\overline w']$ is a sphere class in $Y_0$
		where $\overline w'$ is $w'$ with its domain equipped with the opposite 
		orientation.  We have
		$$
		\omega([u\# \overline w'])= \int_{D^2}u^*\omega -\int_{D^2} w'^*\omega
		= \int_{D^2}u^*\omega
		$$
		where the second integral vanishes since $L_0$ is Lagrangian. Here we note 
		that the right hand side is positive since $u$ is assumed to be
		nonconstant  pseudoholomorphic.
		On the other hand, the left hand side vanishes since $\omega = d\iota^*\lambda$ and hence
		$\omega(u \# \overline w') = 0$ for the sphere map $u \# \overline w'$ whose image is contained in
		$Y \setminus p_{N_{\bb{R}}}^{-1}(B_{\bb{R}}) \subset T^*M_{\bb{R}}$. 
		This implies $ \int_{D^2}u^*\omega = 0$
	        which gives rise to a contradiction and hence the image of $u$ must intersect 
	        $p_{N_{\bb{R}}}^{-1}(B_{\bb{L}})$.  This finishes the proof.
	        	\end{proof}
	
	Note that this lemma has nothing to do with the dimension and holds for any
	Lagrangian multi-sections for any choice of  $J$.

	With the lemma said, we now restrict ourselves to \emph{embedded} Lagrangian multi-sections
	so that $L = i(\widetilde L)$ is a smooth embedded submanifold of $Y$.
	We examine those disks that possibly intersect $p_{N_{\bb{R}}}^{-1}(B_{\bb{L}})$. For this purpose,
	we will apply the moduli intersection argument based on the generic evaluation transversality
	and the dimension counting argument. (See \cite{oh:book1} for the proof of generic
	transversality of the interior evaluation map which can be adapted for the boundary evaluation
	map as done in the more nontrivial case of contact instantons \cite{oh:contacton-transversality}.)

	\begin{remark}\label{rem:choiceJ} Since $L$ is noncompact, we need to examine the $C^0$
estimates in the study of moduli space of $J$-holomorphic disks in general. In this regard,
we would like to attract readers' attention that our SYZ projection
$p_{N_\bR}$ is \emph{not} the cotangent base projection $T^*M_\bR \to M_\bR$ and the multisection $L$ is asymptotic to some cotangent fibers $\{x\}\times N_{\bb{R}}$ of $T^*M_\bR = M_\bR \times N_\bR$.
Therefore if we consider standard almost complex structure $J$ tame to the canonical symplectic form
$\omega = \sum_{j=1}^n d\xi_j \wedge dx_j$, we can apply the monotonicity formula for the $C^0$ estimates as
done in \cite{NZ} in the construction of $A_\infty$ maps for the relevant Fukaya category on $Y$ when we regard
$L$ as an object of the category.
On the other hand, towards the \emph{vertical direction} of the SYZ projection $p_{N_\bR}$, Condition (1) of Definition \ref{def:fuk} of the category $\mathcal C$ takes care of the $C^0$-estimates.
    \end{remark}

	Now we wrap up the proof of Theorem \ref{thm:unobs}.
	\begin{proof}[Proof of Theorem \ref{thm:unobs}] As $\bb{L}$ is graded, 
	it can only bound Maslov index 0 disks. By  the definition of a tautologically unobstructed, 
	it is enough to prove that for a generic choice of $J$, there is no class $\beta \in \pi_2(Y,L)$
	for which the image of an element from the stable map
		compactification $\ol{\cu{M}}_1(L,J;\beta)$ of the moduli space
		\begin{equation}\label{eq:CMbL}
			{\cu{M}}_1(L,J;\beta) : = \{ (u,z) \mid u: z \in \partial D^2,
			D^2 \to Y, \, u(\Int(D^2)) \subset L,\, [u] = \beta, \,  \overline{\partial}_J = 0\}/\sim
		\end{equation}
		intersects $p_{N_{\bb{R}}}^{-1}(B_{\bb{L}})$.
	
	        By the standard index calculation, we have
		 $\text{\rm vir.}\dim \ol{\cu{M}}_1(L,J;\beta) = n+ 1-3 = n-2$ and hence
		$$
		\text{\rm vir.}\dim \partial  \ol{\cu{M}}_1(L,J;\beta) = n-3.
		$$		
		We now examine the image of the evaluation map
		$\ev_0:\ol{\cu{M}}_1(L,J;\beta) \to L$ in $L$. We decompose
		$$
		\ol{\cu{M}}_1(L,J;\beta) = {\cu{M}}_1(L,J;\beta) \sqcup
		 (\ol{\cu{M}}_1(L,J;\beta) \setminus {\cu{M}}_1(L,J;\beta)).
		 $$
		 We also have further decomposition
		 $$
		  (\ol{\cu{M}}_1(L,J;\beta) \setminus {\cu{M}}_1(L,J;\beta)) = 
		  \partial \ol{\cu{M}}_1(L,J;\beta) \sqcup C  \ol{\cu{M}}_1(L,J;\beta) 
		  $$
		  where $C  \ol{\cu{M}}_1(L,J;\beta)$ is the union of higher codimensional strata of 	 
		  $\ol{\cu{M}}_1(L,J;\beta)$. Since there is no non-constant holomorphic sphere in $Y$,
		  all higher codimensional strata consist of more than one disk bubbles, which are
		  easier to rule out. Therefore we will focus on the case of $\partial \ol{\cu{M}}_1(L,J;\beta)$ 
		  and show that no element therefrom can 
		  intersect $p_{N_{\bb{R}}}^{-1}(B_{\bb{L}})$. 
		  
		  For this purpose, we recall that an element therefrom is of the type in
		$$
		\ol{\cu{M}}_1(L,J;\beta_2) \times_{(\ev_0,\ev_1)} \ol{\cu{M}}_2(L,J;\beta_1)
		$$
		with $\beta = \beta_1 + \beta_2$. (Here we follow the notations from \cite{FOOO1} that
		$\ol{\cu{M}}_{k+1}(L,J;\beta)$ is the stable-map compactification of $J$-holomorphic disks
		with $k+1$ marked points enumerated as $(z_0,z_1, \ldots, z_k)$ and $\ev_i$ is the
		corresponding evaluation map at the $i$-th (boundary) marked point.)
		For a generic choice of $J$ for which various relevant evaluation transversalities
		hold,  especially the one for the map $\ev_0: \ol{\cu{M}}_1(L,J;\beta) \to L$ against
		$S_{\bL}=L\cap p_{N_{\bb{R}}}^{-1}(B_{\bb{L}})$ in $L$, we have
		$$
		\dim_{\bb{R}}(\ol{\cu{M}}_1(L,J; \beta)\times_{L}S_{\bb{L}})=(n-2)+(n-2)-n = n-4.
		$$
		Hence  \emph{when $n\leq 3$}, $\ol{\cu{M}}_1(\beta)\times_{L}S_{\bb{L}} =\emptyset$, i.e.,
		no element from $\partial \ol{\cu{M}}_1(L,J; \beta)$ intersects $S_{\bb{L}}$ along the 
		boundary.
		
		Next we examine the intersection at an interior point.	 For this purpose,
		we consider the moduli space $\ol{\cu{M}}_{0,1}(L,J;\beta)$.
		(Here $\ol{\cu{M}}_{0,k}(L,J;\beta)$ is the stable map moduli space of discs with
		$k$ interior points and $0$ boundary marked point.)
		This time we consider the interior evaluation map
		$$
		\ev^+: \ol{\cu{M}}_{0,1}(L,J;\beta) \to Y
		$$
		and consider the fiber product 
		$
		\ol{\cu{M}}_{0,1}(\beta) \times_Y p_{N_{\bb{R}}}^{-1}(B_{\bb{L}}).
		$
		Again a dimension counting provides
		$$
		\dim_{\bb{R}}\left(\ol{\cu{M}}_{0,1}(\beta) \times_Y p_{N_{\bb{R}}}^{-1}(B_{\bb{L}}) \right) 
		= (n-1) +( 2n-2) - 2n = n-3 < 0.
		$$
		Therefore  \emph{ if $n\leq 2$}, we have 
		$$
		\ol{\cu{M}}_{0,1}(\beta) \times_Y p_{N_{\bb{R}}}^{-1}(B_{\bb{L}}) = \emptyset
		$$ 
		and hence the holomorphic disks do not intersect $p_{N_{\bb{R}}}^{-1}$ in the interior
		either.		 Combining the two, we have finished the proof.
	\end{proof}
	
	%\begin{remark} The above proof shows that generically holomorphic disks bounded by $L$ above
	%	do not intersect $i(S_\bL)$ in the interior at all \emph{irrespective of dimension of $Y$.} In fact, for the higher
	%	dimensional cases, the argument used in the above proof
	%	provides an interesting quantum invariant for the SYZ fibration which we will elaborate elsewhere.
	%\end{remark}
	The argument used in the above proof also shows the following corollary.
	
	\begin{corollary}\label{cor:immersed_unobs}
		Suppose $n =2 $. If the immersed sector of a Lagrangian multi-section $\bb{L}$ is concentrated at degree 1, then $\bb{L}$ is tautologically unobstructed.
	\end{corollary}
	\begin{proof}
		Note that $T^*N_{\bb{R}}\backslash i(I_{\bb{L}})$ has trivial second homotopy group. The argument in Theorem \ref{thm:unobs} allows us to rule out the holomorphic disks bounded by $i(\til{L}\backslash I_{\bb{L}})$. As the immersed sector is concentrated at degree 1, by dimension reason, $L$ does not bound any holomorphic disks.
	\end{proof}

	\subsection{Lagrangian multi-sections are mirror to toric vector bundles}
	
	We now use microlocal Morse theory to show that Lagrangian multi-sections are mirror to toric vector bundles, which is our first main result in this paper.
	
	\begin{definition}
		Given a Lagrangian multi-section $\bb{L}$ in $Y$. A point $\xi\in N_{\bb{R}}$ is said to be \emph{regular} if $\xi\in N_{\bb{R}}\backslash B_{\bb{L}}$ and $F_\xi \cap I_\bL = \emptyset$.
		The fiber over a regular point is called a \emph{regular fiber}.
	\end{definition}
	
	Given a graded Lagrangian multi-section $\bb{L}$ of $Y$, a regular fiber $F_{\xi}$ always intersects $L$ transversally. Hence we can talk about the degree of each intersection point. We define the \emph{degree map} $\deg_{\bb{L}}:\til{L}\backslash(p_{\bb{L}}^{-1}(B_{\bb{L}})\cup I_{\bb{L}})\to\bb{Z}$ by
	$$
	\deg_{\bb{L}}:l\mapsto\deg_{F_{p_{\bb{L}}(l)},\bb{L}}(i(l)).
	$$
	This map is continuous as $\theta_i,\theta_{\bb{L}}$ vary continuously in $p$ and $\theta_{F_{\xi}}$ is a constant. In particular, when $\til{L}$ is connected, $\deg_{\bb{L}}$ is constant as $p_{\bb{L}}^{-1}(B_{\bb{L}})\cup I_{\bb{L}}$ is of codimension $\geq 2$. Moreover, around a generic point, the image $L$ is locally Hamiltonian isotopic to a horizontal section $\{x\}\times N_{\bb{R}}$ of the SYZ projection $p_{N_\bR}: Y \to N_\bR$ which can be graded by $0$. As parity doesn't change under isotopy, we see that $\deg_{\bb{L}}\in n+2\bb{Z}$.
	
	\begin{definition}
		Let $\bb{L}$ be a graded Lagrangian multi-section of $p_{N_{\bb{R}}}:Y\to N_{\bb{R}}$. The \emph{canonical grading of $\bb{L}$} is the grading so that $\deg_{\bb{L}}=n$.
	\end{definition}
	
	For a Lagrangian multi-section $\bb{L}$ of $p_{N_{\bb{R}}}:Y\to N_{\bb{R}}$, we would like to compute
	$$
	\Hom_{\cu{F}uk_{\bb{K}}^0(Y)}(D_{m,-\sigma}[-n],\bb{L})
	$$
	for $m\in M$ and $\sigma\in\Sigma(n)$. Let $\sigma\in\Sigma$ be maximal and $\xi\in\Int(\sigma)$ be regular. We choose a proper path $\xi:[0,\infty)\to\Int(\sigma)$ starting from $\xi$ so that $-\xi(t)$ is regular for all $t\geq 0$ and
	$$\lim_{t\to\infty}\|\xi(t)\|=\infty.$$
	Such a path always exists because $B_{\bb{L}}\cup p_{\bb{L}}(I_{\bb{L}})$ are of codimension at least 2. Let $p_{-\xi(t)}\in F_{-\xi(t)}\cap L$ and consider its limit
	$$p_{-\xi}^{\infty}:=\lim_{t\to\infty}p_{-\xi(t)}\in M\times\Int(-\sigma)^{\infty}\subset\Lambda_{\Sigma}^{\infty}.$$
	Put
	$$m(p_{-\xi}):=p_{M_{\bb{R}}}(p_{-\xi}^{\infty})\in M.$$
	Although $p_{-\xi}^{\infty}$ depends on the choice of the path $\xi(t)$, the lattice element $m(p_{-\xi})$ does not; it only depends on the intersection point $p_{-\xi}\in F_{-\xi}\cap L$. It is obvious that
	$$\Hom_{\cu{F}uk_{\bb{K}}^0(Y)}(D_{m,-\xi}[-n],\bb{L})\not\simeq 0$$
	only when $m=m(p_{-\xi})$ for some $p_{-\xi}\in F_{-\xi}\cap L$ because we can always move the linking disk closer to $\Lambda_{\Sigma}^{\infty}$.

	\begin{theorem}\label{thm:degree0}
		If $(\bb{L},E)\in \cu{F}uk_{\bb{K}}^0(Y,\Lambda_{\Sigma})$ is a canonically graded Lagrangian multi-section, then for any $m\in M$ and $\sigma\in\Sigma(n)$, the Floer complex $\Hom_{\cu{F}uk_{\bb{K}}^0(Y)}(D_{m,-\sigma}[-n],(\bb{L},E))$ is concentrated at degree 0.
	\end{theorem}
	\begin{proof}
		We only need to consider the case $E$ is trivial. Let $\sigma\in\Sigma(n)$ and $m\in M$. Fix a regular point $-\xi_0\in\Int(\sigma)$ with $\|\xi_0\|>R_m$ and $m(p_{-\xi_0})=m$. Let $D_{m,-\xi_0}$ be a representative of the linking disk $D_{m,-\sigma}$ so that $D_{m,-\xi_0}\cap F_{-\xi_0}$ projects to an open ball $U_{m,-\xi_0}\subset M_{\bb{R}}$ centered at $m$. Then there exists $R_{m,-\xi_0}>0$ such that the open subset (of $L$)
		$$\bb{L}_{>R_{m,-\xi_0}}:=\{p_{-\xi}\in L:-\xi\in\Int(\sigma) \text{ is regular},\|\xi\|>R_{m,-\xi_0},m(p_{-\xi})=m\}$$
		is an embedded Lagrangian submanifold of $T^*U_{m,-\xi_0}$ (see Figure \ref{fig:linking_disk})
		\begin{figure}[H]
			\centering
			\includegraphics[width=80mm]{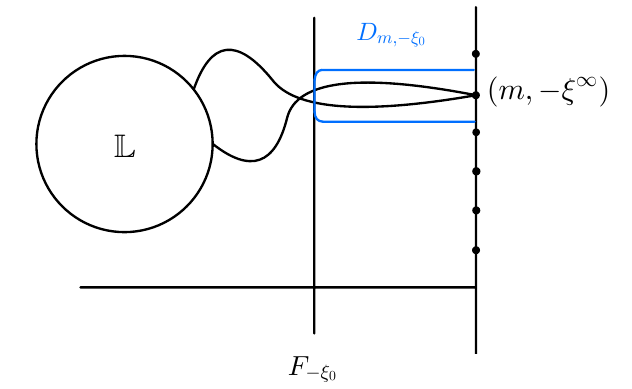}
			\caption{The linking disk $D_{m,-\xi_0}$ agrees with $F_{-\xi_0}$ on $T^*U_{m,-\xi_0}$.}
			\label{fig:linking_disk}
		\end{figure}
		By the definition and the choice of $D_{m,-\xi_0}$, it
		agrees with $F_{-\xi_0}$ on $T^*U_{m,-\xi_0}$ and hence
		$$
		D_{m,-\xi_0}\cap L=F_{-\xi_0}\cap T^*U_{m,-\xi_0}\cap L.
		$$
		With respect to the canonical grading, each intersection point of $F_{-\xi_0}$ and $L$ is of degree $n$. 
As the canonical grading of $D_{m,-\xi_0}$ is the same as $F_{-\xi_0}$ whenever they overlap, $\Hom_{\cu{F}uk_{\bb{K}}^0(Y)}(D_{m,-\xi_0}[-n],\bb{L})$ is concentrated at degree 0.
	\end{proof}
	
	By combining Theorem \ref{thm:degree0}, equivariant CCC \cite{CCC},
	HMS \cite{CCC_HMS} for toric varieties and Theorem 1.2 in the introduction, we obtain the following corollary.
	
	\begin{corollary}\label{cor:MS_gives_bundle}
		Let $\Sigma$ be a complete fan. The mirror of a canonically graded $\Lambda_{\Sigma}$-admissible Lagrangian multi-section is quasi-isomorphic to a toric vector bundle over $X_{\Sigma}$.
	\end{corollary}
	
	\begin{remark}
		Let $\bb{L}$ be an $\Lambda_{\Sigma}$-admissible Lagrangian multi-section and $\cu{E}_{\bb{L}}$ its mirror. At each torus fixed point $X_{\sigma}\subset X_{\Sigma}$, the proof of Theorem \ref{thm:degree0} and Theorem 1.7 in \cite{Morse_theory_TVB} allow us to deduce that
		$$(\cu{E}_{\bb{L}}|_{X_{\sigma}})_m\cong H^0(\mu_{m,-\sigma}(\cu{F}_{\bb{L}}))\cong HF^0(D_{m,-\xi}[-n],\bb{L}),$$
		for all $\xi\in\Int(\sigma)$ and $m\in M$. In particular, the degree of the branched covering map $p_{\bb{L}}:L\to N_{\bb{R}}$ equals to the rank of $\cu{E}_{\bb{L}}$. Hence we think of the link disk $D_{m,-\sigma}$ as the ``SYZ-fiber of $p_{N_{\bb{R}}}:Y\to N_{\bb{R}}$ at infinity" which is equivariantly
  mirror to the sky-scrapper sheaf at the torus fixed point $X_{\sigma}$ with 
  the equivariant structure determined by $m$.
	\end{remark}

	\section{Lagrangian realization problem}\label{sec:TGRC}
	
	We now introduce the Lagrangian realization problem. To do so, let's first recall the notion of a 
\emph{tropical Lagrangian multi-sections over a fan $\Sigma$} \cite{branched_cover_fan, Suen_trop_lag}.

    \begin{definition}
    Let $\Sigma$ be a complete fan on $N_{\bb{R}}$. An \emph{$r$-fold tropical Lagrangian multi-section over $\Sigma$} is a data $\bb{L}^{\trop}:=(L^{\trop},\Sigma_{L^{\trop}},\mu,p,\varphi^{\trop})$, where
    \begin{enumerate}
        \item $(L^{\trop},\Sigma_{L^{\trop}})$ is a connected cone complex and $p:(L^{\trop},\Sigma_{L^{\trop}})\to(N_{\bb{R}},\Sigma)$ is a continuous surjection between connected cone complexes which maps cones in $\Sigma_{L^{\trop}}$ homomorphically onto cones in $\Sigma$.
        \item $\mu:\Sigma_{L^{\trop}}\to\bb{Z}_{>0}$ is a map (called the multiplicity map) such
         that $$
         \sum_{\sigma'\subset p^{-1}(\sigma)}\mu(\sigma')=r,
         $$
        for all $\sigma\in\Sigma$.
        \item $\varphi^{\trop}:L^{\trop}\to\bb{R}$ is a continuous function so that for any 
        cone $\sigma'\in\Sigma_{L^{\trop}}$, there exists $m(\sigma')\in M$ (regarded as linear function on $N_{\bb{R}}$) 
        such that $\varphi^{\trop}|_{\sigma'}=m(\sigma')\circ p|_{\sigma'}$.
    \end{enumerate}
    A cone $\tau'\in\Sigma_{L^{\trop}}$ is said to be \emph{ramified} if $\mu(\tau')>1$. %A tropical Lagrangian is said to be \emph{separated} if for any ray $\rho\in\Sigma(1)$ and distinct lifts $\rho^{(\alpha)},\rho^{(\beta)}\in\Sigma_{L^{\trop}}(1)$ of $\rho$, we have $\varphi^{\trop}|_{\rho^{(\alpha)}}\neq \varphi^{\trop}|_{\rho^{(\beta)}}$.
    \end{definition}

    In \cite{branched_cover_fan}, Payne associated a tropical Lagrangian multi-section $\bb{L}_{\cu{E}}^{\trop}$ 
    to each toric vector bundle $\cu{E}$ by considering its equivariant structure on each toric affine chart. As the name suggested, we also think of them as the ``tropicalization" of Lagrangian multi-sections in the following sense. We think of the piecewise linear function $\varphi^{\trop}$ as a ``potential function" of a Lagrangian multi-section that is roughly defined as
    $$\{(p_*(d\varphi^{\trop}(l)),p(l))\in M_{\bb{R}}\times N_{\bb{R}}:l\in L^{\trop}\}.$$
    Of course, this multi-section is far from being well-defined as both $p$ and $\varphi^{\trop}$ are not smooth.

    In this section, we try to address the following problem.
	
	\begin{question}
		Given a tropical Lagrangian multi-section $\bb{L}^{\trop}$ over a complete fan $\Sigma$, is there a (tautologically) unobstructed Lagrangian multi-section $\bb{L}\in \cu{F}uk_{\bb{K}}^0(Y,\Lambda_{\Sigma})$ whose mirror $\cu{E}_{\bb{L}}\in\cu{P}er_T(X_{\Sigma})$ has its associated tropical Lagrangian multi-section $\bb{L}^{\trop}$?
	\end{question}
	
	We call this the \emph{Lagrangian realization problem}. Of course, one cannot expect that such a mirror 
exists for all $\bb{L}^{\trop}$. As was shown in Example 5.1 in \cite{Suen_trop_lag}, there is a 2-fold tropical Lagrangian multi-section that does not arise from toric vector bundles even in dimension 2. Hence we don't expect it to have the corresponding A-side object. This implicates that we need to restrict ourselves to some special tropical Lagrangian multi-sections. (See  \cite{arguz_real_lag_CY3, Hicks_realization} for the related works.)
	
	\subsection{The $N$-generic condition}\label{sec:tangent}

	We restrict ourselves to 2-fold tropical Lagrangian multi-sections over a complete 2-dimensional fan. Topologically, there are two possibilities for such a 2-fold covering.
	\begin{itemize}
		\item [(O)] If $p:L^{\trop}\to N_{\bb{R}}$ is a maximal covering, then $L^{\trop}$ is homeomorphic to $\bb{R}^2$, and $p:L^{\trop}\to N_{\bb{R}}$ is topologically identified with the square map $z\mapsto z^2$ on $\bb{C}$. Therefore, the subset $C:=p^{-1}(S^1)\subset L^{\trop}$ is a circle. (See \cite[Definition 2.26 \& Section 2.5]{branched_cover_fan} for the definition of \emph{maximal covering}.)
		\item [(E)] If $p:L^{\trop}\to N_{\bb{R}}$ is not maximal, then we can pass to a maximal covering $p_{max}:L_{max}^{\trop}\to N_{\bb{R}}$ which is topologically the trivial covering map $N_{\bb{R}}\sqcup N_{\bb{R}}\to N_{\bb{R}}$. Therefore, the subset $C:=p^{-1}(S^1)\subset L^{\trop}$ are two disjoint circles $C^{(1)},C^{(2)}$.
	\end{itemize}

In any case, the deck group of the projection $C \to S^1$ is isomorphic to $\bb{Z}/2\bb{Z}$ 
and hence we have a unique non-trivial deck transformation $\gamma: C \to C$ with $\gamma^2 = id$. In Case (O), if we parametrize the circle $C$ by $[0,2\pi)$, 
then $\gamma$ maps $\theta$ to $\theta+\pi \text{ mod }2\pi$.
	
	\begin{definition}\label{def:N_generic}
		Let $\bb{L}^{\trop}=(L^{\trop},\Sigma_L,\varphi^{\trop})$ be a 2-fold tropical Lagrangian multi-section over a 2-dimensional complete fan $\Sigma$. We say $\bb{L}^{\trop}$ is \emph{$N$-generic} if it satisfies one of the following conditions.
		\begin{enumerate}
			\item [(O)] If $\bb{L}$ is maximal, the graph of the function $\varphi^{\trop}|_{[0,\pi)}$ intersects that of $\varphi^{\trop}\circ\gamma|_{[0,\pi)}$ transversely at $N$ smooth points.
			\item [(E)] If $\bb{L}$ is not maximal, the graph of the function $\varphi^{\trop}|_{C^{(1)}}$ intersects that of $\varphi^{\trop}|_{C^{(2)}}$ transversely at $N$ smooth points.
		\end{enumerate}
	\end{definition}
	
	\begin{remark}
		The number $N$ is completely determined by a given tropical Lagrangian multi-section $\bb{L}^{\trop}$. ``Generic" here means the graphs of $\varphi^{\trop},\varphi^{\trop}\circ\gamma$ do not intersect at corners. It can be easily shown that the parity of $N$ depends only on the topology of $L^{\trop}$, 
		namely, we have $N$ is odd in Case (O) and $N$ is even in Case (E).
	\end{remark}
	
	The reason why we consider $N$-generic objects is  the following relationship with
	the realizability of \emph{tropical Lagrangian multi-section} by rank 2 toric vector bundles.
	
	\begin{proposition}\label{prop:N>2} Let $\bb{L}^{\trop}$ be a maximal 2-fold tropical Lagrangian 
	multi-section  over a complete 2-dimensional fan $\Sigma$. Then it can be realized by a rank 2 toric vector bundle over $X_{\Sigma}$ if and only if $\bb{L}^{\trop}$ is $N$-generic with $N\geq 3$.
	\end{proposition}
	\begin{proof}  This is a consequence of a combination of Proposition 3.23 
	and Theorem 5.9 in  \cite{Suen_trop_lag}. More specifically we have the following 
	which is the 2-dimensional case of \cite[Proposition 3.23]{Suen_trop_lag}.
	\begin{lemma} Let $\dim N_\bR = 2$. Then the following two are equivalent:
	\begin{enumerate}
	\item  $\bL^{\text{\rm trop}}$ is combinatorially  indecomposable and maximal.
	\item  $\bL^{\text{\rm trop}}$ is 1-separated  and the ramification locus
	of $p: \bL^{\text{\rm trop}} \to N_{\bR}$ lies in the codimension 2 strata of $(L, \Sigma_L)$.
	\end{enumerate}
	\end{lemma}	
	This lemma shows that the graph of the function $\varphi^{\trop}|_{[0,\pi)}$
	 intersects that of $\varphi^{\trop}\circ\gamma|_{[0,\pi)}$ only at smooth points 
	 and the intersections are transverse. Then
	 \cite[Theorem 5.9]{Suen_trop_lag}  implies $\bb{L}^{\trop}$ satisfies 
	 the slope condition (\cite[Definition 5.8]{Suen_trop_lag}): The latter
	  is equivalent to $N\geq 3$ by the intermediate value theorem.
	\end{proof}
	
	\begin{remark}
		The realization problem is trivial for Case (E) because we can easily construct two Lagrangian sections by smoothing the two piecewise linear functions suitably. However, if we ask ourselves whether $\bb{L}^{\trop}$ can be realized by \emph{embedded object}, the question becomes non-trivial and this is what we are going to address.
	\end{remark}

	\begin{example}
		On a smooth projective surface $X_{\Sigma}$, the associated tropical Lagrangian multi-section of the tangent bundle $T_{X_{\Sigma}}$ is $\#\Sigma(1)$-generic because it takes the value 1 on a lift of $v_{\rho}$ and the value 0 on the other lift. Note that when $X_{\Sigma}$ is complete, $\#\Sigma(1)\geq 3$.
	\end{example}
	
	At this moment, $L^{\trop}$ carries no smooth structure and thus we cannot talk about smoothing of $p$. Nevertheless, $L^{\trop}$ has a smooth structure away from a neighborhood of the minimal cone. 
	
\subsection{Smoothing of tropical Lagrangian multi-sections}

We first outline our strategy of smoothing $L^{\trop}$.  We decompose 
$L^{\trop}$ into the union of the region outside a sufficiently large disk in $N_{\bR}$, 
one inside some smaller disk and the intermediate region in between the two.
We call each of them \emph{the outer, the inner and the annular part respectively.}

\subsubsection{Smoothing of the outer part}

On the complement of the disk $D_R:=\{\|\xi\|\leq R\}$ with $R>0$ sufficiently large, we define
	$$
	C_{\geq R}:=\begin{cases}
	\bb{C}\backslash D_{\sqrt{R}} & \text{ if }N \text{ is odd},\\
	(\bb{C}\backslash D_R)\sqcup(\bb{C}\backslash D_R) & \text{ if }N \text{ is even},
	\end{cases}
	$$
	and identify $N_{\bb{R}}$ with $\bb{C}$ with complex coordinate $\xi=\xi_1+\sqrt{-1}\xi_2$. Let $h:C_{\geq R}\to L^{\trop}\backslash p^{-1}(D_R)$ be a homeomorphism so that 
$$
p\circ h=p_{std}: C_{\geq R} \subset \bC \to N_{\bR} \cong \bC
$$
where $p_{std}$ is the standard 2-fold branched covering map given by the formula
$$
p_{std}(l): = 
	\begin{dcases} l^2 & \text{ if }N \text{ is odd},\\
	l & \text{ if }N\text{ is even}.
	\end{dcases}
$$
We first construct a Lagrangian multi-section $\bb{L}_{\varphi_{\geq R}}$ over $N_{\bb{R}}\backslash D_R$ with asymptotic condition at infinity determined by $\bb{L}^{\trop}$ as follows. The tropical Lagrangian multi-section $\bb{L}^{\trop}$ determines a conical Lagrangian subset
	$$
	\Lambda_{\bb{L}^{\trop}}:=\bigcup_{\tau'\in\Sigma_L}m(\tau')\times(-\tau)\subset\Lambda_{\Sigma},
	$$
	where we regard the slope $m(\tau')$ of $\varphi^{\trop}$ along $\tau'\in\Sigma_L$ as a coset of $M_{\bb{R}}$. In order to avoid the sign issue arising from the minus sign in the $N_{\bb{R}}$-coordinate of $\Lambda_{\bb{L}^{\trop}}$ and $\Lambda_{\Sigma}$, we work with
    $$
    \alpha(\Lambda_{\bb{L}^{\trop}})=\bigcup_{\tau'\in\Sigma_L}m(\tau')\times\tau\subset\alpha(\Lambda_{\Sigma}),
    $$
    where $\alpha:Y\to Y$ is the involution $(x,\xi)\mapsto(x,-\xi)$. After we construct our Lagrangian multi-section with asymptotics in $\alpha(\Lambda_{\bb{L}^{\trop}})^{\infty}$, we simply apply $\alpha$ again 
     to obtain a Lagrangian multi-section with asymptotics back in $\Lambda_{\bb{L}^{\trop}}^{\infty}$.
	
	\begin{lemma}\label{lem:infinity_part}
		Suppose $\bb{L}^{\trop}$ is $N$-generic. Then for each $R>0$, there exists a smoothing $\varphi_{\geq R}$ of $\varphi^{\trop}\circ h$ such that the outer part
		$$
		\bb{L}_{\varphi_{\geq R}}
		:= \left\{\left((p_{std}^*)^{-1}d(\varphi_{\geq R}(l)),p_{std}(l)\right)
		\in M_{\bb{R}}\times N_{\bb{R}}:l\in C_{\geq R}\right\}
		$$
		is an embedded Lagrangian multi-section over $N_{\bb{R}}\backslash D_R$ and satisfies $\bb{L}_{\varphi_{\geq R}}^{\infty}\subset\alpha(\Lambda_{\bb{L}^{\trop}})^{\infty}$.
	\end{lemma}
	\begin{proof}
		We abuse the notation $\varphi^{\trop}$ for $\varphi^{\trop}\circ h$ for the
		notational convenience. 
		Let $(r,\theta)$ be the polar coordinate of $C_{\geq R}$. We shall focus on the odd $N$ case.
  
        We regard the map $\theta\mapsto\varphi^{\trop}(r,\theta)$ as a $2\pi$-periodic function 
        on $\bb{R}$. As $\varphi^{\trop}$ is piecewise linear, we can write
		\begin{equation}\label{eqn:radial_formula_varphi_R}
		    \varphi^{\trop}|_{\sigma^{(i)}}(r,\theta)=r^2\left(a_i\cos(2\theta)+b_i\sin(2\theta)\right)=:r^2f(\theta),
		\end{equation}
        for some $a_i,b_i\in\bb{Z}$ such that $m(\sigma^{(i)})=(a_i,b_i)$ and so $(a_1,b_1)\neq (a_2,b_2)$ 
        by the $N$-genericity. Let $\alpha_i\in[0,2\pi]$ be the angles where $\varphi^{\trop}(r,-)$ is not smooth. 
        Note that $\alpha_i$'s are all independent of $r$, and $N$-genericity implies
        $$f(\alpha_i)\neq f(\alpha_i+\pi)$$
        for all $i$. Hence there exists $\delta>0$ such that
        $$
        f(\theta)\neq f(\theta+\pi)
        $$
        for all $\theta\in\bigcup_i[\alpha_i-\delta,\alpha_i+\delta]$ and so $|f(\theta)-f(\theta+\pi)|$ receives a positive lower bound on $\bigcup_i[\alpha_i-\delta,\alpha_i+\delta]$. For this $\delta>0$, 
        we can choose $m_{\delta}>0$ so small that
        \begin{equation}\label{eqn:f_lower_bound}
            |f(\theta)-f(\theta+\pi)|\geq m_{\delta}|f(\alpha_i)-f(\alpha_i+\pi)|
        \end{equation}
        for all $\theta\in[\alpha_i-\delta,\alpha_i+\delta]$ and for all $i$. Using the same $\delta>0$, let $\phi_{\delta}:\bb{R}\to\bb{R}_{\geq 0}$ be a non-negative smooth function such that
        $$
        \phi_{\delta}(\theta)=\begin{dcases}
            1 & \text{ for }|\theta|\leq\frac{\delta}{2},\\
            0 & \text{ for }|\theta|\geq\delta. 
        \end{dcases}
        $$
        The convex combination
		$$\phi_{\delta}(r(\theta-\alpha_i))\, m_{\delta}\varphi^{\trop}(r,\alpha_i)
		+(1-\phi_{\delta}(r(\theta-\alpha_i))\, \varphi^{\trop}(r,\theta)$$
		is a smooth $2\pi$-periodic function on $\bb{R}$. We can shrink $\delta>0$ so that the support of $\phi_{\delta}(r(\theta-\alpha_i))$'s are all disjoint for all $r\geq R$. Therefore, we can apply this construction to every $\alpha_i$ separately and obtain a smooth function $\varphi_{\geq R}:C_{\geq R}\to\bb{R}$.
        Since each smoothing neighbourhood shrinks to the corresponding rays
		$\{\theta = \alpha_i\}$ as $r\to\infty$, we have
		\begin{equation}\label{eqn:asymptotic_condition}
			\lim_{r\to\infty}(p_{std}^*)^{-1}(d\varphi_{\geq R}(r,\theta))=m(\tau'),
		\end{equation}
		if $(r,\theta)\in\tau' \in \Sigma_L$. Define
		\begin{equation}\label{eqn:lag}
			\bb{L}_{\varphi_{\geq R}}:=\{((p_{std}^*)^{-1}d(\varphi_{\geq R}(l)),p_{std}(l))\in M_{\bb{R}}\times N_{\bb{R}}:l\in C_{\geq R}\}.
		\end{equation}
		It is clear that $\bb{L}_{\varphi_{\geq R}}$ is a multi-section over $N_{\bb{R}}\backslash D_R$. Moreover, it follows from \eqref{eqn:asymptotic_condition} that $\bb{L}_{\varphi_{\geq R}}^{\infty}\subset\alpha(\Lambda_{\bb{L}^{\trop}})^{\infty}$.
		
		It remains to show that $\bb{L}_{\varphi_{\geq R}}$ is embedded. A straightforward calculation shows that being embedded is equivalent to
		$$
		\pd{\varphi_{\geq R}}{r}(r,\theta)\neq\pd{\varphi_{\geq R}}{r}(r,\theta+\pi)\text{ or }\pd{\varphi_{\geq R}}{\theta}(r,\theta)\neq\pd{\varphi_{\geq R}}{\theta}(r,\theta+\pi),
		$$
		for all $(r,\theta)\in[R,\infty)\times[0,2\pi]$. We have
		\begin{align*}
			\pd{}{\theta}\varphi^{\trop}\circ h|_{\sigma^{(i)}}=&\,r^2f'(\theta)=2r^2\left(-a_i\sin(2\theta)+b_i\cos(2\theta)\right),\\
			\pd{}{r}\varphi^{\trop}\circ h|_{\sigma^{(i)}}=&\,2rf(\theta)=2r\left(a_i\cos(2\theta)+b_i\sin(2\theta)\right).
		\end{align*}
		An elementary calculation shows that there are no solutions to
		$$
		\begin{dcases}
		f(\theta)=f(\theta+\pi)\\
		f'(\theta)=f'(\theta+\pi)
		\end{dcases}
		$$
		Hence there are no solutions to
		$$
		\begin{dcases}
		\pd{\varphi_{\geq R}}{r}(r,\theta)=\pd{\varphi_{\geq R}}{r}(r,\theta+\pi)\\
		\pd{\varphi_{\geq R}}{\theta}(r,\theta)=\pd{\varphi_{\geq R}}{\theta}(r,\theta+\pi)
		\end{dcases}
		$$
		outside the smoothing neighbourhood. 
		Within the smoothing neighbourhood $[\alpha_i-\frac{\delta}{r},\alpha_i+\frac{\delta}{r}]$, 
		note that
        \begin{eqnarray*}
        &{}& 
            \pd{\varphi_{\geq R}}{r}(r,\alpha_i)-\pd{\varphi_{\geq R}}{r}(r,\alpha_i+\pi) \\
            & =&\,2r\big(\phi(r(\theta-\alpha_i))m_{\delta}(f(\alpha_i)-f(\alpha_i+\pi)) 
             +(1-\phi(r(\theta-\alpha_i))(f(\theta)-f(\theta+\pi)))\big)\\
            &{}& \quad +r^2\phi'(r(\theta-\alpha_i)) \cdot (\theta-\alpha_i)\cdot 
            \big(m_{\delta}(f(\alpha_i)-f(\alpha_i+\pi))
            -(f(\theta)-f(\theta+\pi))\big)
        \end{eqnarray*}
        Since $\phi'(r(\theta-\alpha_i))(\theta-\alpha_i)\leq 0$, by \eqref{eqn:f_lower_bound}, we have
        $$f(\alpha_i)-f(\alpha_i)>0 \,(\text{resp.}<0) \Longrightarrow \pd{\varphi_{\geq R}}{r}(r,\alpha_i)-\pd{\varphi_{\geq R}}{r}(r,\alpha_i+\pi)>0 \,(\text{resp.}<0),$$
        for all $\theta\in[\alpha_i-\delta,\alpha_i+\delta]$. This proves embeddedness in the smoothing neighbourhoods. As a whole, $\bb{L}_{\varphi_{\geq R}}$ must be embedded. The even case is similar and therefore omitted.
	\end{proof}

The proof of Lemma \ref{lem:infinity_part} actually gives the following lemma, which will be used later.

\begin{lemma}\label{lem:ineq_relation_varphi_R}
    For any $R>0$, the smoothing $\varphi_{\geq R}$ of $\varphi^{\trop}\circ h$ constructed in Lemma \ref{lem:infinity_part} satisfies the inequality relation
    $$\varphi_{\geq R}(l^{(1)})-\varphi_{\geq R}(l^{(2)})>0\Longrightarrow\pd{\varphi_{\geq R}}{r}(l^{(1)})-\pd{\varphi_{\geq R}}{r}(l^{(2)})>0,$$
    for all $|l|\geq R$.
\end{lemma}

\subsubsection{Surgery of the inner part: hyper-elliptic local model}
	
Now we need to  glue back the inner part over $D_R\subset N_{\bb{R}}$
after smoothing it.
 To do so, we need a local model of Lagrangian multi-section over $D_R$ with a well-controlled immersed sector 
 to guarantee unobstructedness, the explanation of which is now in order.

 When $N\geq 3$, such a local model is determined by the number $N$ as follows.
	
	\begin{itemize}
		\item [(O)] If $N=2k+1>1$, we put $g=k-1$. Consider the hyper-elliptic curve
		$$L_{f_{2g+1}}:=\{(x,\xi)\in M_{\bb{R}}\times N_{\bb{R}}:x^2=f_{2g+1}(\xi)\},$$
		where the complex structure on $M_{\bb{R}}\times N_{\bb{R}}$ is given by $x:=x_1-\sqrt{-1}x_2,\xi:=\xi_1+\sqrt{-1}\xi_2$ and
		$$f_{2g+1}(\xi):=a_{2g+1}(\xi^{2g+1}+a_{2g}\xi^{2g}+\cdots+a_1\xi)\in\bb{R}[\xi],$$
		is a polynomial with roots having multiplicity at most 2 and $a_d>0$. The second projection $p_{N_{\bb{R}}}:L_{f_{2g+1}}\to N_{\bb{R}}$ is a 2-fold branched covering map. The underlying topological surface of $L_{f_{2g+1}}$ has arithmetic genus $g$ and 1 puncture.
		\item [(E)] If $N=2k+2>2$, we put $g=k-1$. Consider the hyper-elliptic curve
		$$L_{f_{2g+2}}:=\{(x,\xi)\in M_{\bb{R}}\times N_{\bb{R}}:x^2=f_{2g+2}(\xi)\}$$
		where
		$$f_{2g+2}(\xi):=a_{2g+2}(\xi^{2g+2}+a_{2g+1}\xi^{2g+1}+\cdots+a_1\xi)\in\bb{R}[\xi]$$
		is a polynomial with roots having multiplicity at most 2 and $a_{2g+2}>0$. The second projection $p_{N_{\bb{R}}}:L_{f_{2g+2}}\to N_{\bb{R}}$ is a 2-fold branched covering map.  The underlying topological surface of $L_{f_{2g+2}}$ has arithmetic genus $g$ and 2 punctures.
	\end{itemize}

    \begin{remark}
        See Remark \ref{rem:direct_sum} for the reason why we allow roots of multiplicity 2.
    \end{remark}
	
	For $d\in\{2g+1,2g+2\}$, as $L_{f_d}$ is a holomorphic curve, it becomes a (special) Lagrangian immersion after hyper-K\"ahler rotation. Clearly $L_{f_d}$ is a 2-fold multi-section of $p_{N_{\bb{R}}}:Y\to N_{\bb{R}}$, but it doesn't have any control at the infinity $Y^{\infty}$.
	
	From now on, we assume $N\geq 3$. We want to glue $L_{f_d}$ with $\bb{L}_{\varphi_{\geq R}}$ along their boundaries, so we need to write $L_{f_d}$ as in the form of (\ref{eqn:lag}). 
	However, due to the non-trivial topology of $L_{f_d}$, it may not be possible. But luckily, we don't need 
	the coordinate description (\ref{eqn:lag}) to hold everywhere, we just need it to hold on the cylindrical ends of $L_{f_d}$. To do so, we consider the complex coordinate functions $x(l),\xi(l)$ of $L_{f_d}$. Let $d\in\{2g+1,2g+2\}$. In any case, we can regard $\sqrt{f_d}$ as a single-valued holomorphic function on $L_{f_d}$. We want to solve for a function $\varphi$ so that
	$$(p_{N_{\bb{R}}}^*)^{-1}(d\varphi(l))=(x_1(l),x_2(l)),$$
	which can be written as
	$$d\varphi={\rm{Re}}(\sqrt{f_d}p_{N_{\bb{R}}}^*d\xi).$$
	This leads us to consider the holomorphic 1-form
	$$\alpha_{f_d}:=\sqrt{f_d}p_{N_{\bb{R}}}^*d\xi.$$
	Let $R_0>0$. Note that $N$ and $d$ have the same parity, so
	$$C_{\geq R_0}=\begin{dcases}
	\bb{C}\backslash D_{\sqrt{R_0}} & \text{ if }d \text{ is odd},\\
	(\bb{C}\backslash D_{R_0})\sqcup(\bb{C}\backslash D_{R_0}) & \text{ if }d \text{ is even},
	\end{dcases}$$
	Let $F:C_{\geq R_0}\to L_{f_d}$ be a holomorphic embedding so that $F$ maps $C_{\geq R_0}$ to the cylindrical ends of $L_{f_d}$ and $p_{N_{\bb{R}}}\circ F=p_{std}$. We have
	$$F^*\alpha_{f_d}=\begin{dcases}
	2l\sqrt{f_d(l^2)}dl  & \text{ if }d\text{ is odd},\\
	\pm\sqrt{f_d(l)}dl & \text{ if }d\text{ is even}.
	\end{dcases}$$
	For $R_0>0$ large enough, for any $|l|>R_0$, we have the power series expansion
	$$F^*\alpha_{f_d}=\begin{dcases}
	2\sqrt{a_d}l^{d+1}\left(1+\sum_{i=1}^{\infty}c_il^{-2i}\right)dl  & \text{ if }d\text{ is odd},\\
	\pm \sqrt{a_d}l^{\frac{d}{2}}\left(1+\sum_{i=1}^{\infty}c_il^{-i}\right)dl & \text{ if }d\text{ is even},
	\end{dcases}$$
	for some constant $c_i\in\bb{R}$. By integrating term-by-term, we get a multi-valued primitive $\varphi_{f_d}^{\bb{C}}$ of $F^*\alpha_{f_d}$, which has the form
	$$\varphi_{f_d}^{\bb{C}}(l):=\begin{dcases}
	\sqrt{a_d}l^{d+2}\left(c_0+\sum_{i=1}^{\infty}c_il^{-2i}\right)  & \text{ if }d\text{ is odd},\\
	\pm \sqrt{a_d}l^{\frac{d}{2}+1}\left(c_0+\sum_{\substack{i\geq 1\\i\neq\frac{d}{2}+1}}c_il^{-i}\right)+ c_{\frac{d}{2}+1}\log(l) & \text{ if }d\text{ is even}
	\end{dcases}$$
	after absorbing constants into $c_i$. As $f_d$ has real coefficients $c_i\in\bb{R}$ for all $i\geq 0$. Therefore, the real part $\varphi_{f_d}$ of $\varphi_{f_d}^{\bb{C}}$ is single-valued\footnote{The condition $f_d\in\bb{R}[\xi]$ just for convenience. It works perfectly well for $f_d\in\bb{C}[\xi]$ as long as $\mathrm{Im}(c_{\frac{d}{2}+1})=0$ when $d$ is even.} and hence a primitive of ${\rm{Re}}(F^*\alpha_{f_d})$. Using the polar coordinate $(r,\theta)$ on $C_{\geq R_0}$, this primitive is given by
	\begin{equation}\label{eqn:radial_formula_varphi_std}
	\varphi_{f_d}(r,\theta)=
	\begin{dcases}
	\sqrt{a_d}\sum_{i=0}^{\infty} c_ir^{d+2-2i} \cos((d+2-2i)\theta) & \text{ if } d\text{ is odd},\\
	\pm\sqrt{a_d}\sum_{\substack{i\geq 0\\i\neq\frac{d}{2}+1}}c_ir^{\frac{d}{2}+1-i}\cos\left(\left(\frac{d}{2}+1-i\right)\theta\right)+c_{\frac{d}{2}+1}\log(r) & \text{ if }d\text{ is even}.
	\end{dcases}
 \end{equation}
 We denote by $\varphi_{f_d}^{\pm}$ the two disjoint branches of $\varphi_{f_d}$ in the even case. Motivated by the definition of the $N$-genericity condition, we should look at solutions to the equation
    
	$$
	\begin{dcases}
	    \varphi_{f_d}(r,\theta)=\varphi_{f_d}(r,\theta+\pi),\,\theta\in [0,\pi) & \text{ if }d\text{ is odd},\\
        \varphi_{f_d}^{+}(r,\theta)=\varphi_{f_d}^{-}(r,\theta),\,\theta\in [0,2\pi) & \text{ if }d\text{ is even}.
	\end{dcases}
	$$

    Note that zeros of $\theta\mapsto\varphi_{f_d}(r,\theta)$ in $[0,\pi)$ (when $d$ is odd) or $[0,2\pi)$ (when $d$ is even) are precisely the solution to the above equations. We compare the zeros of $\varphi_{f_d}$ and its leading order term. When $r>R_0$, it is easy to show that $\varphi_{f_d}$ and its angular derivative satisfy the following estimates
    \begin{equation}\label{eqn:ineqn_odd}
        \begin{dcases}
        \left|\frac{1}{\sqrt{a_d}r^{d+2}}\varphi_{f_d}(r,\theta)-c_0\cos((d+2)\theta)\right|\leq C_{R_0}r^{-2},\\
        \left|\frac{1}{\sqrt{a_d}r^{d+2}}\pd{\varphi_{f_d}}{\theta}(r,\theta)+(d+2)c_0\sin((d+2)\theta)\right|\leq C_{R_0}r^{-2}
    \end{dcases}
    \end{equation}
    when $d$ is odd and
    \begin{equation}\label{eqn:ineqn_even}
        \begin{dcases}
        \left|\frac{1}{\sqrt{a_d}r^{\frac{d}{2}+1}}\varphi_{f_d}(r,\theta)-c_0\cos\left(\left(\frac{d}{2}+1)\theta\right)\right)\right|\leq C_{R_0}r^{-1},\\
        \left|\frac{1}{\sqrt{a_d}r^{\frac{d}{2}+1}}\pd{\varphi_{f_d}}{\theta}(r,\theta)+\left(\frac{d}{2}+1\right)c_0\sin\left(\left(\frac{d}{2}+1)\theta\right)\right)\right|\leq C_{R_0}r^{-1}
    \end{dcases}
    \end{equation}
    when $d$ is even, for some constant $C_{R_0}>0$ independent of $a_d$.

    \begin{lemma}\label{lem:solutions}
        When $d$ is odd (resp. even), there exists $R>R_0$ such that for any $r>R$, the function $\theta\mapsto\varphi_{f_d}(r,\theta)$ has $d+2$ many distinct zeros $\theta_i^d(r)\in[0,\pi)$ (resp. $[0,2\pi)$), $i=1,\dots,d+2$. Moreover, the zeros $\theta_i^d(r)$ vary smoothly in $r$ for all $i=1,\dots,d+2$. 
    \end{lemma}
    \begin{proof}
        We only consider the odd case. The even case follows by exactly the same argument. Choose $R>1$ large enough so that $2C_{R_0}R^{-2}<c_0$. Then
        $$
        c_0\left(-\frac{1}{2}+\cos((d+2)\theta)\right)<\frac{1}{\sqrt{a_d}r^{d+2}}\varphi_{f_d}(r,\theta)
        <c_0\left(\frac{1}{2}+\cos((d+2)\theta)\right),
        $$
        for all $(r,\theta)$ with $r>R$. Note that the equation
        $$
        \cos((d+2)\theta))=\pm \frac12
        $$
        has $d+2$ solutions in $(0,\pi)$, say $\theta_i^{\pm}$, $i=1,\dots,d+2$, so
        $\varphi_{f_d}(r,\theta)>0$ whenever $\theta\in [0,\theta_1^+]\sqcup[\theta_2^+,\theta_3^+]\sqcup\dots\sqcup[\theta_{d+1}^+,\theta_{d+2}^+]$ and $\varphi_{f_d}(r,\theta)<0$ whenever $\theta\in[\theta_1^-,\theta_2^-]\sqcup[\theta_3^-,\theta_4^-]\sqcup\dots\sqcup[\theta_{d+2}^-,\pi)$. By 
        the intermediate value theorem, $\varphi_{f_d}$ must have zeros in $(\theta_1^+,\theta_1^-),(\theta_2^-,\theta_2^+),\dots,(\theta_{d+2}^+,\theta_{d+2}^-)$. The angular derivative estimate in \eqref{eqn:ineqn_odd} implies $\pd{\varphi_{f_d}}{\theta}$ has a definite sign on each of these open intervals and hence the \emph{solution} must be unique in these intervals. Hence there are exactly $d+2$ of them. 
         Denote these solutions by $\theta_i^d(r)$, $i=1,\dots,d+2$. As the intersection between 
         the graph of $\varphi_{f_d}(r,-)$ and the $\theta$-axis is transverse for all $r>R$,  
         each of these solutions varies smoothly in $r$ by the implicit function theorem.
    \end{proof}
	We list the solutions $\theta_i^d(r)$ guaranteed by Lemma \ref{lem:solutions} in the increasing order
$\theta_i^d(r)<\theta_{i+1}^{std}(r)$. Note that while the constant $R$ and these solutions depend on the constants $c_i$'s
 they are all independent of the coefficient $a_d$. By a suitable choice of $a_d$, we would like
 to rescale $L_{f_d}$ so that it becomes almost of the ``same size" as that of $\bb{L}_{\varphi_{\geq R}}$ 
 along the intermediate annular region so that we can glue
  the outer part $\bb{L}_{\varphi_{\geq R}}$ and $L_{f_d}$ along the 
  annular part without creating extra immersed points. In other words, we want the immersed sector 
  of the resulting Lagrangian multi-section to lie over the points in $\bC$ that are exactly the same 
  as the roots of the polynomial $f_d$ which have multiplicity $2$. 
	
	Let $\theta_i$ be the angles that solve the equation
	\begin{equation}\label{eqn:theta_i}
	    \begin{dcases}
	    \varphi_{\geq R}(r,\theta_i)=\varphi_{\geq R}(r,\theta_i+\pi), \theta_i\in [0,\pi) & \text{ when }d\text{ is odd},\\
        \varphi_{\geq R}^{+}(r,\theta_i)=\varphi_{\geq R}^{-}(r,\theta_i), \theta_i\in [0,2\pi) & \text{ when }d\text{ is even}.
	\end{dcases}
	\end{equation}
By our construction of $\varphi_{\geq R}$ in Lemma \ref{lem:infinity_part}, these are also the solutions to $\varphi^{\trop}\circ h=\varphi^{\trop}\circ h\circ\gamma$.
     By the $N$-genericity, there are $N=d+2$ of them. We assume $\theta_i<\theta_{i+1}$, for all $i=1,\dots,d+1$. The key observation is that $\varphi_{\geq R}$ and $\varphi_{f_d}$ share the same $N$-genericity condition. Our idea is to glue $\varphi_{\geq R}$ and $\varphi_{f_d}$ (hence $\bb{L}_{\varphi_{\geq R}}$ and $L_{f_d}$) along 
    the intermediate annular region so that the angle $\theta_i$ matches with $\theta_i^d(r)$ for all $i=1,\dots,d+2$. We therefore need to modify  $\varphi_{f_d}$ so that its zeros become exactly the same as $\theta_i$'s.
     We refer readers to Figure \ref{fig:rho} and \ref{fig:glue} for some geometric intuition.
    
    %We focus on the following region of $L_{f_d}$:
    %$$L_{\geq R}:=\{((p_{std}^{-1})^*(d\varphi_{f_d}(l)),p_{std}(l))\in M_{\bb{R}}\times N_{\bb{R}}:|l|\geq R\}.$$
    To begin, we look at the asymptotics of the roots $\theta_i^d(r)$ as $r\to\infty$.

    \begin{lemma}\label{lem:bound_theta_std}
        Let $R>R_0>0$ be in Lemma \ref{lem:solutions}. Then there exists a constant $M_{R_0}>0$, independent of $a_d>0$ and $R>0$, such that the solutions $\theta_i^d(r)$ satisfy the inequality
        $$
        \left|\pd{\theta_i^d}{r}(r,\theta)\right|\leq\begin{dcases}
            \frac{M_{R_0}}{r^3} & \text{ if }d\text{ is odd},\\
            \frac{M_{R_0}}{r^2} & \text{ if }d\text{ is even},
        \end{dcases}
        $$
        for all $r>R$ and $\theta\in[0,2\pi]$.
    \end{lemma}
    \begin{proof}
        Differentiate $\varphi_{f_d}(r,\theta_i^d(r))=0$ with respect to $r$, we obtain
        $$\pd{\theta_i^d}{r}=\frac{\dpsum{i\geq 0}{}c_ir^{d+1-2i}\cos((d+2-2i)\theta_i^d(r))}{\dpsum{i\geq 0}{}c_ir^{d+2-2i}\sin((d+2-2i)\theta_i^d(r))}.$$
        By \eqref{eqn:ineqn_odd}, we have
        \begin{align*}
            \left|\cos((d+2)\theta_i^d(r))\right|\leq &\,\frac{C_{R_0}}{r^2}.\\
            \left|\sin((d+2)\theta_i^d(r))\right|\geq &\,C_{R_0}',
        \end{align*}
        for some constant $C_{R_0},C_{R_0}'>0$ which is independent of $a_d$ and $R>0$. Then it is easy to obtain
        $$\left|\pd{\theta_i^d}{r}(r,\theta)\right|\leq\frac{M_{R_0}}{r^3},$$
        for all $r>R$ and $\theta\in[0,2\pi]$. The even case is similar by applying \eqref{eqn:ineqn_even}.
    \end{proof}
    
    Let $R>R_0>0$ be in Lemma \ref{lem:solutions} and $\varepsilon\in(0,1)$ be fixed so that $R-\varepsilon>R_0$. Lemma \ref{lem:bound_theta_std} motivates us to consider a smooth function $\rho:[R-\varepsilon,\infty)\times\bb{R}\to\bb{R}$ with the following properties:
    \begin{itemize}
    \item We have
    \begin{align*}
        \rho(r,\theta+\pi)=&\,\rho(r,\theta)+\pi \text{ when } d \text{ is odd},\\
        \rho(r,\theta+2\pi)=&\,\rho(r,\theta)+2\pi \text{ when } d \text{ is even}
    \end{align*}
    \item $\rho$ satisfies
   $$
	    \rho(r,\theta) =
    \begin{dcases}
	    \theta_i^d(r) & \text{ if } r\geq R+\varepsilon\text{ and }\theta=\theta_i,\\
        \theta & \text{ if }r\in[R-\varepsilon,R],
	\end{dcases}
$$
\item and
\begin{eqnarray*}
    \pd{\rho}{\theta}(r,\theta)& > &\,0\,,\forall (r,\theta)\in[R-\varepsilon,\infty)\times\bb{R},\\
    \left|\pd{\rho}{r}(r,\theta)\right|& \leq &\,
    \begin{dcases}
        \frac{M_{R_0}}{r^3}, \forall\,(r,\theta)\in[R+\varepsilon,\infty)\times\bb{R}, \text{ if }d\text{ is odd},\\
        \frac{M_{R_0}}{r^2}, \forall\, (r,\theta)\in[R+\varepsilon,\infty)\times\bb{R}, \text{ if }d\text{ is even}
    \end{dcases}
     \end{eqnarray*}
    for some constant $M_{R_0}>0$ independent of $a_d,R>0$. See Figure \ref{fig:rho_function}.
\end{itemize}

    \begin{figure}[H]
		\centering
		\includegraphics[width=60mm]{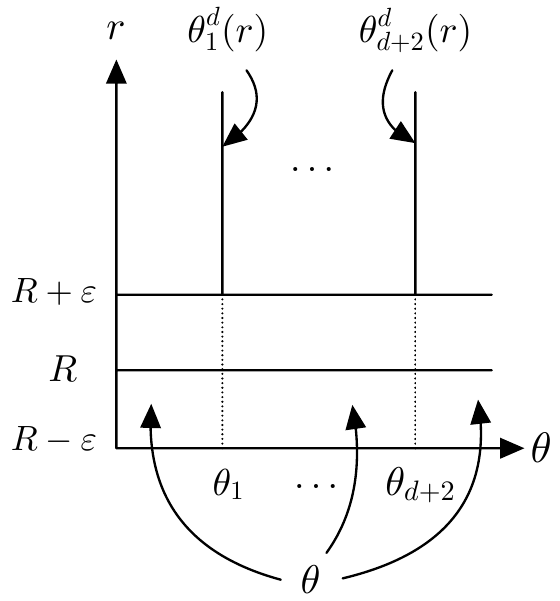}
		\caption{The function $(r,\theta)\mapsto\rho(r,\theta)$. The curved arrows indicate the value of $\rho(r,\theta)$ at different regions. The $r$-derivative decay at a rate $r^{-3}$ (resp. $r^{-2}$) after $r\geq R+\varepsilon$. We extend $\rho$ to the whole $\bb{R}$ by declaring that $\rho(r,\theta+\pi)=\rho(r,\theta)+\pi$ in the odd case and $\rho(r,\theta+2\pi)=\rho(r,\theta)+2\pi$ in the even case.}
        \label{fig:rho_function}
	\end{figure}
    As $\pd{\rho}{\theta}(r,\theta)>0$, for all $r\geq R-\varepsilon$ and $\theta\in\bb{R}$, we get a diffeomorphism $F_{\rho}:C_{\geq R-\varepsilon}\to C_{\geq R-\varepsilon}$ defined by
	$$
F_{\rho}:(r,\theta)\mapsto(r,\rho(r,\theta))
$$
	that carries $(r,\theta_i)$ to $(r,\theta_i^d(r))$, for all $i=1,\dots,d+2$. Define
	$$
\varphi_{f_d}^{\rho}(r,\theta):=(\varphi_{f_d}\circ F_{\rho})(r,\theta)=\varphi_{f_d}(r,\rho(r,\theta)).
$$
    As $\rho(r,\theta)=\theta$ for all $r\in[R-\varepsilon,R]$, the embedded Lagrangian
	$$L_{\geq R-\varepsilon}^{\rho}:=\{(p_{std}^*)^{-1}(d\varphi_{f_d}^{\rho}(l)),p_{std}(l))\in M_{\bb{R}}\times N_{\bb{R}}:|l|\geq R-\varepsilon\}$$
	now agrees with $L_{f_d}$ when $|l|\in[R-\varepsilon,R]$ and therefore can be extended to an immersed Lagrangian multi-section $L_{f_d}^{\rho}\subset Y$ with the immersed sector being given by roots of $f_d$. See Figure \ref{fig:rho}.
 
    \begin{figure}[H]
		\centering
		\includegraphics[width=130mm]{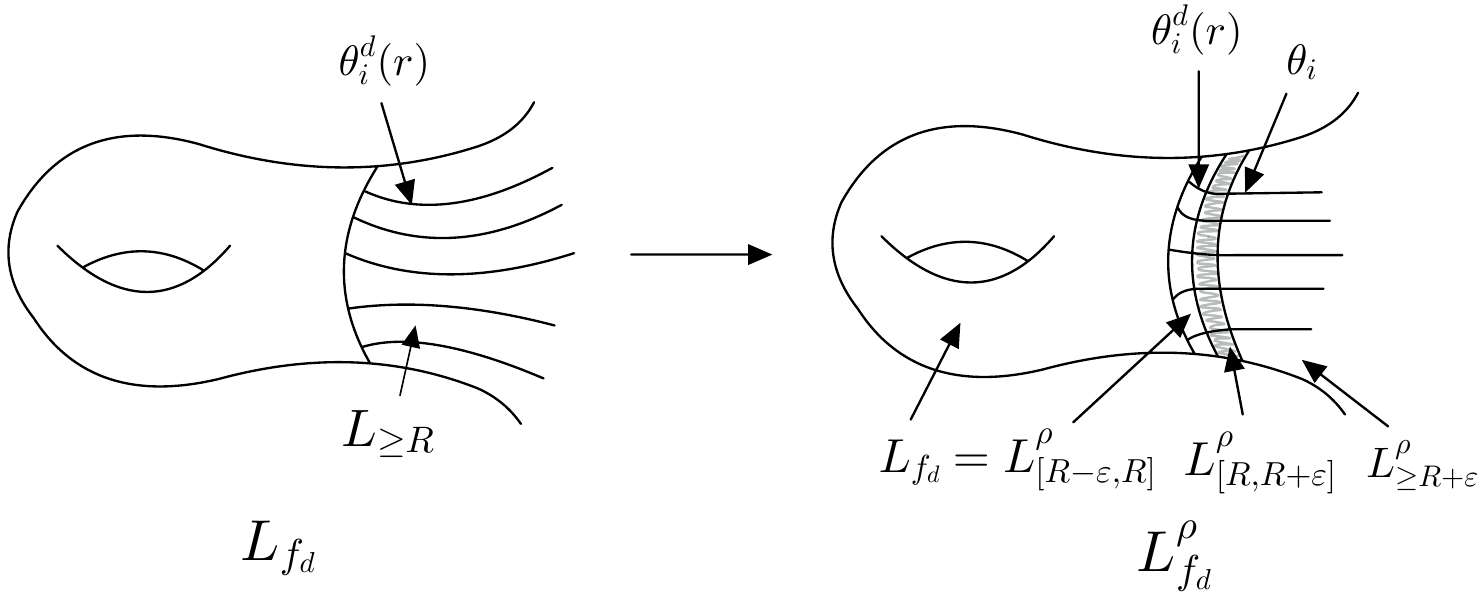}
		\caption{We modifying $L_{\geq R}$ by $\rho$ so that the angles $\theta_i^d(r)$'s match with $\theta_i$'s.}
		\label{fig:rho}
	\end{figure}

\subsubsection{The intermediate annular part: gluing}
 
	Now, we glue $L_{f_d}^{\rho}$ with $\bb{L}_{\varphi_{\geq R}}$ along the annulus $C_{[R+\varepsilon,R+1]}$. 
To do so, we first need to choose the branch of $\varphi_{f_d}$ so that $\varphi_{f_d}^{\rho}$ ``looks like" $\varphi_{\geq R}$ in the following sense. By applying the involution $(x,\xi)\mapsto(-x,\xi)$ to $L_{f_d}$ if necessary, we may assume
	\begin{equation}\label{eqn:same_sign}
	    \varphi_{f_d}^{\rho}(l^{(1)})-\varphi_{f_d}^{\rho}(l^{(2)})>0 \Longleftrightarrow\varphi_{\geq R}(l^{(1)})-\varphi_{\geq R}(l^{(2)})>0.
	\end{equation}
    See Figure \ref{fig:same_sign} for the Case (O). As the intersections are transversal, for each $\theta_i$, there exists a small open interval $I_i$ of $\theta_i$ such their angular derivatives on $I_i$ also satisfy a similar inequality relation:
    \begin{equation}\label{eqn:angular_same_sign}
		\begin{dcases}
		    \pd{}{\theta}(\varphi_{f_d}^{\rho}(r,\theta)-\varphi_{f_d}^{\rho}(r,\theta+\pi))>0 \Longleftrightarrow\pd{}{\theta}(\varphi_{\geq R}(r,\theta)-\varphi_{\geq R}(r,\theta+\pi))>0 & \text{ if }d\text{ is odd},\\
            \pd{}{\theta}(\varphi_{f_d}^{\rho,+}(r,\theta)-\varphi_{f_d}^{\rho,-}(r,\theta))>0 \Longleftrightarrow\pd{}{\theta}(\varphi_{\geq R}^{+}(r,\theta)-\varphi_{\geq R}^{-}(r,\theta))>0 & \text{ if }d\text{ is even},
		\end{dcases}
	\end{equation}
    for all $\theta\in I_i$ and for all $i=1,\dots,d+2$. Analogously to Lemma \ref{lem:ineq_relation_varphi_R}, 
    we have the following inequality relations for the radial derivatives of $\varphi_{f_d}^{\rho}$. 
    \begin{figure}[H]
		\centering
		\includegraphics[width=130mm]{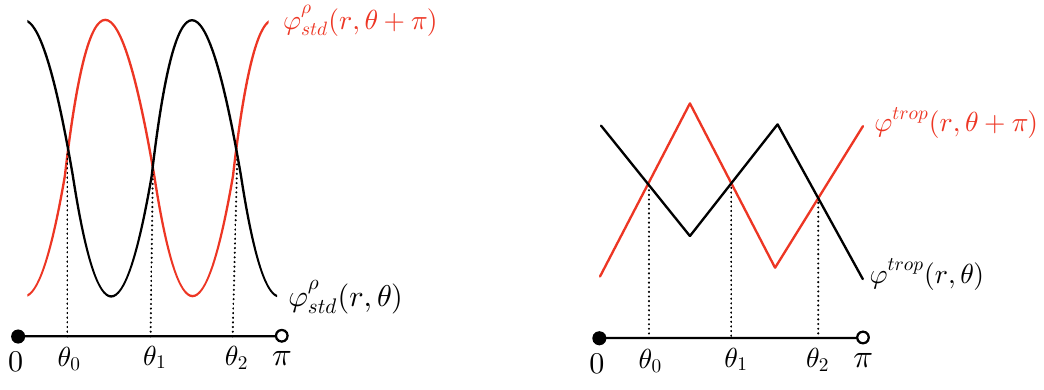}
		\caption{This is Case (O), where $N=3$. As $\varphi_{\geq R}$ is obtained by rounding codimension 1 corners of $\varphi^{\trop}$, we may assume, by changing the branch of $\varphi_{f_d}$ if needed, that is, by adding $\pi$ to $\theta$, the differences $\varphi_{f_d}^{\rho}(r,\theta)-\varphi_{f_d}^{\rho}(r,\theta+\pi)$ and $\varphi_{\geq R}(r,\theta)-\varphi_{\geq R}(r,\theta+\pi)$ have same sign everywhere on $[0,\pi)$.}
        \label{fig:same_sign}
	\end{figure}
    \begin{lemma}\label{lem:ineq_relation_varphi_std}
        There exists $R>0$, independent of $a_d$, such that for any small open intervals $I_i$ of $\theta_i$, we have
        $$\begin{dcases}
            \varphi_{f_d}^{\rho}(r,\theta)-\varphi_{f_d}^{\rho}(r,\theta+\pi)>0\, (<0) \Longrightarrow\pd{\varphi_{f_d}^{\rho}}{r}(r,\theta)-\pd{\varphi_{f_d}^{\rho}}{r}(r,\theta+\pi)>0\, (<0) & \text{ if }d\text{ is odd},\\
            \varphi_{f_d}^{\rho,+}(r,\theta)-\varphi_{f_d}^{\rho,-}(r,\theta)>0\, (<0) \Longrightarrow\pd{\varphi_{f_d}^{\rho,+}}{r}(r,\theta)-\pd{\varphi_{f_d}^{\rho,-}}{r}(r,\theta)>0\, (<0) & \text{ if }d\text{ is even},
        \end{dcases}$$
        for all $r>R+\varepsilon$ and $\theta\notin \bigcup_{i=1}^{d+2}I_i$
    \end{lemma}
    \begin{proof}
        Again, we only consider the odd case. Note that
        $$
        \varphi_{f_d}^{\rho}(r,\theta)-\varphi_{f_d}^{\rho}(r,\theta+\pi)=2\varphi_{f_d}^{\rho}(r,\theta).
        $$
        Thus, we need to show that
        $$
        \varphi_{f_d}^{\rho}(r,\theta)>0\,(<0)\Longrightarrow\pd{\varphi_{f_d}^{\rho}}{r}(r,\theta)>0\,(<0).
        $$
        Let $R_0>0$ be large enough so that we can find solutions $\theta_i^d(r)$, $r>R_0$ as in Lemma \ref{lem:solutions}. For each small open interval $I_i$ of $\theta_i$, we have
        $$
        \frac{1}{\sqrt{a_d}r^{d+2}}\varphi_{f_d}^{\rho}(r,\theta)\neq 0
        $$
        for all $r>R_0$ and $\theta\in K$, where
        $$
        K:=[0,2\pi]\Big\backslash\bigcup_{i=1}^{d+2}I_i.
        $$
        Thus there exists $M_{R_0}>0$, independent of $a_d$, such that
        $$
        \left|\frac{1}{\sqrt{a_d}}\varphi_{f_d}^{\rho}(r,\theta)\right|\geq r^{d+2}M_{R_0},
        $$
        for all $(r,\theta)\in [R_0,\infty)\times K$. Combining this with \eqref{eqn:ineqn_odd}, we derive
        $$
        |\cos((d+2)\rho(r,\theta))|\geq M_{R_0}-C_{R_0}r^{-2},
        $$
        which is strictly positive for all $r>R$ if $R>R_0$ is chosen to be large enough. We have
        \begin{align*}
            \frac{1}{\sqrt{a_d}}\pd{\varphi_{f_d}^{\rho}}{r}(r,\theta)=&\,\sum_{i\geq 0}c_i(d+2-2i)r^{d+1-2i}\cos((d+2-2i)\rho(r,\theta))\\
            &-\sum_{i\geq 0}c_i(d+2-2i)r^{d+2-2i}\sin((d+2-2i)\rho(r,\theta))\pd{\rho}{r}.
        \end{align*}
        Recall that $\pd{\rho}{r}=O(r^{-3})$, we then have
        $$
        \left|\frac{1}{\sqrt{a_d}}\pd{\varphi_{f_d}^{\rho}}{r}(r,\theta)\right|\geq M_{R_0}'r^{d+1}-M_{R_0}''r^{d-1},
        $$
        for some constants $M_{R_0}',M_{R_0}''>0$. We can therefore take an even larger $R>R_0$ to conclude the right hand
         side is positive for all $r>R$. The even case is similar and we just present the result that we obtain:
        $$
        \left|\frac{1}{\sqrt{a_d}}\pd{}{r}(\varphi_{f_d}^{\rho,+}(r,\theta)-\varphi_{f_d}^{\rho,-}(r,\theta))\right|\geq M_{R_0}'r^{\frac{1}{2}+1}-M_{R_0}''r^{\frac{1}{2}},
        $$
        for some constants $M_{R_0}',M_{R_0}''>0$, independent of $a_d$.
    \end{proof}  
    Now, let $\chi_{R,\varepsilon}:[0,\infty)\to\bb{R}$ be a smooth function such that
	$$\chi_{R,\varepsilon}|_{[0,R+\varepsilon]}\equiv 0,\,\chi_{R,\varepsilon}|_{[R+1,\infty)}\equiv 1,\,\chi_{R,\varepsilon}'\geq 0.$$
	Define $\varphi_{\bb{L}^{\trop}}:C_{\geq R-\varepsilon}\to\bb{R}$ by
	$$\varphi_{\bb{L}^{\trop}}(r,\theta):=\chi_{R,\varepsilon}(r)\varphi_{\geq R}+(1-\chi_{R,\varepsilon}(r))\varphi_{f_d}^{\rho}(r,\theta).$$
	It follows from the choice of $\chi_{R,\varepsilon}$ that
	$$\varphi_{\bb{L}^{\trop}}=\begin{dcases}
	\varphi_{\geq R} & \text{ if }r\geq R+1,\\
	\varphi_{f_d}^{\rho} & \text{ if }r\in[R-\varepsilon,R+\varepsilon].
	\end{dcases}$$
	Then we obtain a Lagrangian multi-section (see Figure \ref{fig:glue} again)
	$$\alpha(\bb{L}_{\varphi_{\bb{L}^{\trop}}}):=\begin{dcases}
    \bb{L}_{\varphi_{\geq R}} & \text{ over }p_{std}(C_{\geq R+1}),\\
	\{((p_{std}^*)^{-1}(d\varphi_{\bb{L}^{\trop}}(l)),p_{std}(l))\in M_{\bb{R}}\times N_{\bb{R}}\backslash D_R:l\in C_{[R+\varepsilon,R+1]}\} & \text{ over }p_{std}(C_{[R+\varepsilon,R+1]}),\\
	L_{f_d}^{\rho} & \text{ over } D_{R+\varepsilon}
	\end{dcases}$$
	that satisfies the asymptotic condition $\bb{L}_{\varphi_{\bb{L}^{\trop}}}^{\infty}\subset\Lambda_{\bb{L}^{\trop}}^{\infty}$. We now prove that by choosing $a_d>0$ small enough, the immersed sectors of $\bb{L}_{\varphi_{\bb{L}^{\trop}}}$ are exactly given by the roots of $f_d$ that have multiplicity 2. This will follow once we show that $\bb{L}_{\varphi_{\bb{L}^{\trop}}}$ is embedded on the annulus $C_{[R+\varepsilon,R+1]}$.
    \begin{figure}[H]
		\centering
		\includegraphics[width=150mm]{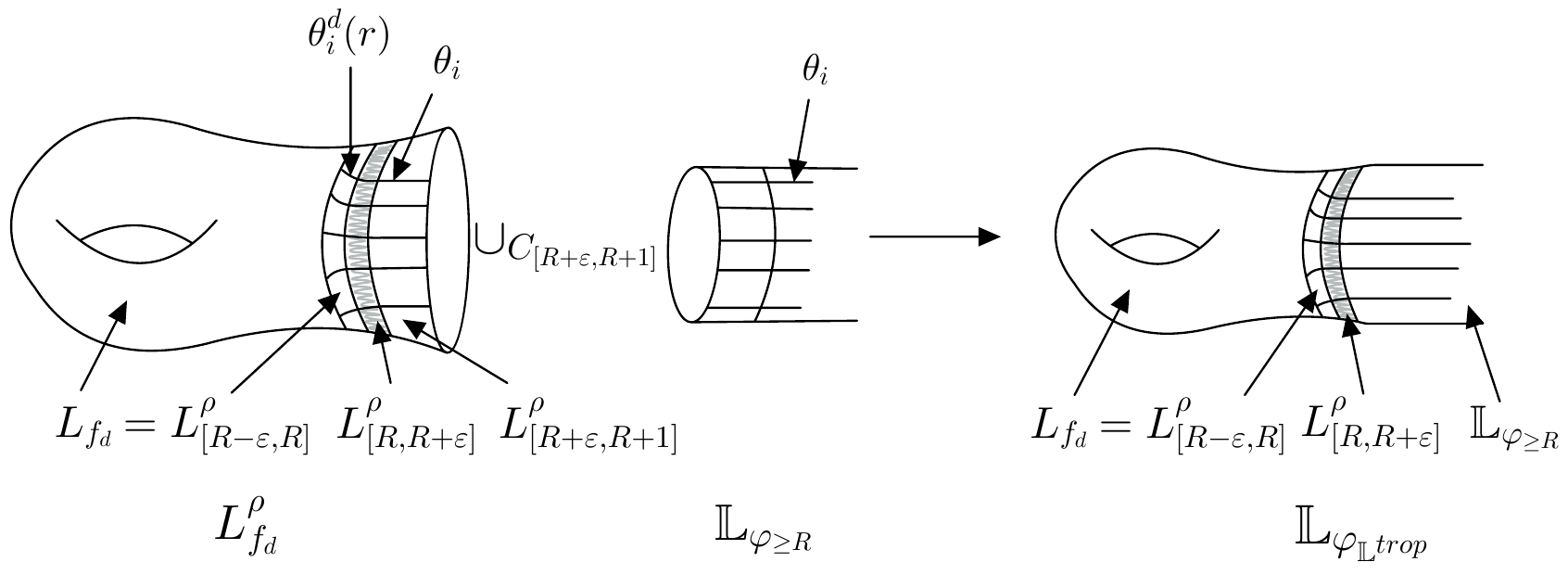}
		\caption{We glue $\bb{L}_{\varphi_{\geq R}}$ and $L_{f_d}^{\rho}$ along the cylinder $C_{[R+\varepsilon,R+1]}$ by shrinking $a_d>0$ to ensure embeddedness on $C_{[R+\varepsilon,R+1]}$.}
		\label{fig:glue}
	\end{figure}

\subsection{Control of embeddedness of the gluing annulus}

	Fix $R>0$ so that Lemma \ref{lem:solutions} and Lemma \ref{lem:ineq_relation_varphi_std} hold. Keep in mind that everything so far is independent of $a_d$. Let's consider the odd case. Embeddedness fails if and only if there exists $(r,\theta)\in C_{[R+\varepsilon,R+1]}$ such that
	$$
	\pd{\varphi_{\bb{L}^{\trop}}}{r}(r,\theta)=\pd{\varphi_{\bb{L}^{\trop}}}{r}(r,\theta+\pi)\text{ and }\pd{\varphi_{\bb{L}^{\trop}}}{\theta}(r,\theta)=\pd{\varphi_{\bb{L}^{\trop}}}{\theta}(r,\theta+\pi).
	$$
	The angular equation can be written as
	\begin{equation}\label{eqn:angular}
		\chi_{R,\varepsilon}(r)\left(\pd{\varphi_{\geq R}}{\theta}(r,\theta)-\pd{\varphi_{\geq R}}{\theta}(r,\theta+\pi)\right)=(1-\chi_{R,\varepsilon}(r))\left(\pd{\varphi_{f_d}^{\rho}}{\theta}(r,\theta+\pi)-\pd{\varphi_{std }^{\rho}}{\theta}(r,\theta)\right).
	\end{equation}
	For the angles $\theta_i$'s so that $\varphi_{\geq R}(r,\theta_i)=\varphi_{\geq R}(r,\theta_i+\pi)$, let $I_i$ be a small open interval of $\theta_i$ so that the inequality relation \eqref{eqn:angular_same_sign} holds. Then the left-handed side of \eqref{eqn:angular} always has the opposite sign to the right-handed side on $[R+\varepsilon,R+1]\times I_i$ for all $i=1,\dots,d+2$, and so \eqref{eqn:angular} fails to hold on $[R+\varepsilon,R+1]\times\bigcup_{i=1}^{d+2}I_i$. Let
	$$
	K:=[0,2\pi]\Big\backslash\bigcup_{i=0}^dI_i,
	$$
	which is of course compact. Note that
	$$|\varphi_{\geq R}(r,\theta)-\varphi_{\geq R}(r,\theta+\pi)|>0$$
	on the compact subset $[R+\varepsilon,R+1]\times K$ and hence receiving a positive lower bound. As $R$ and $\varepsilon$ are independent of $a_d$, we can then choose $a_d>0$ small enough so that
	\begin{equation}\label{eqn:large_a_d}
	    |\varphi_{f_d}^{\rho}(r,\theta)-\varphi_{f_d}^{\rho}(r,\theta+\pi)|\leq|\varphi_{\geq R}(r,\theta)-\varphi_{\geq R}(r,\theta+\pi)|,
	\end{equation}
	for all $(r,\theta)\in [R+\varepsilon,R+1]\times K$. Now, the radial equation can be written as
	\begin{align}\label{eq:radial-eq}
		&\chi_{R,\varepsilon}(r)\left(\pd{\varphi_{\geq R}}{r}(r,\theta)-\pd{\varphi_{\geq R}}{r}(r,\theta+\pi)\right)+(1-\chi_{R,\varepsilon}(r))\left(\pd{\varphi_{f_d}^{\rho}}{r}(r,\theta)-\pd{\varphi_{f_d}^{\rho}}{r}(r,\theta+\pi)\right) \nonumber\\
		=&\,\chi_{R,\varepsilon}'(r)\left(\varphi_{\geq R}(r,\theta+\pi)-\varphi_{\geq R}(r,\theta)+\varphi_{f_d}^{\rho}(r,\theta)-\varphi_{f_d}^{\rho}(r,\theta+\pi)\right).
	\end{align}
	Recall the inequality relations \eqref{eqn:same_sign}, Lemma \ref{lem:ineq_relation_varphi_R} and \ref{lem:ineq_relation_varphi_std}. Together with the inequality \eqref{eqn:large_a_d}, it is now straightforward to check that the left-handed side of \eqref{eq:radial-eq} always has the opposite sign to the right-handed side on $[R+\varepsilon,R+1]\times K$. As a whole, we conclude that $\bb{L}_{\varphi_{\bb{L}^{\trop}}}$ must also be embedded on $C_{[R+\varepsilon,R+1]}$. This gives embeddedness of $\bb{L}_{\varphi_{\bb{L}^{\trop}}}$ on $C_{[R+\varepsilon,R+1]}$ in the even case. The same argument works perfectly in the even case and is thus omitted.

    Combining the above, we now conclude our construction of Lagrangian multi-sections which will be mirror
    to toric vector bundles.

	\begin{theorem}\label{thm:unobs_immersed_Lag}
		The Lagrangian immersion $\bb{L}_{\varphi_{\bb{L}^{\trop}}}$ is a spin, graded, tame 2-fold multi-section with $\bb{L}_{\varphi_{\bb{L}^{\trop}}}^{\infty}\subset\Lambda_{\bb{L}^{\trop}}^{\infty}$ and there is a 1-1 correspondence between the immersed double points of $\bb{L}_{\varphi_{\bb{L}^{\trop}}}$ and the roots of $f_d$ with multiplicity 2. The immersed sector of $\bb{L}_{\varphi_{\bb{L}^{\trop}}}$ is concentrated at degree 1. In particular, $\bb{L}_{\varphi_{\bb{L}^{\trop}}}$ is tautologically unobstructed.
	\end{theorem}
	\begin{proof}
		Being spin follows from the vanishing of $H^2(\til{L};\bb{Z}/2\bb{Z})$. For being graded, note that the local model $L_{f_d}$ is a holomorphic curve, so ${\rm{Im}}(\Omega|_{L_{f_d}})={\rm{vol}}_{L_{f_d}}$. Hence the phase class $[\theta]\in H^1(\til{L};U(1))$ vanishes on the compact part of $\til{L}$. Since $H_1(\til{L};\bb{Z})$ is generated by the cycles of the compact part, we have $\int_{\gamma}\theta=0$ for all $\gamma\in H_1(\til{L};\bb{Z})$. This implies $[\theta]=0$ and so $\bb{L}_{\varphi_{\bb{L}^{\trop}}}$ can be graded. Tameness follows from the fact that the closure $\ol{\iota(\bb{L}_{\varphi_{\bb{L}^{\trop}}})}\subset D^*M_{\bb{R}}$ is diffeomorphic to $\bb{L}_{\varphi_{\bb{L}^{\trop}}}^{\infty}\times(r,\infty]$ around the infinity $S^*M_{\bb{R}}$. Then the argument provided in the proof of Lemma 5.4.5 in \cite{NZ} gives tameness.
		
		The conclusion of unobstructedness follows from Corollary \ref{cor:immersed_unobs}.
		 It remains to show that the immersed sector of $\bb{L}_{\varphi_{\bb{L}^{\trop}}}$ is concentrated at degree 1. At an immersed point $p\in L$, it has the local model
		$$L_+\cup L_-=\{x=\xi\cdot g(\xi)\}\cup\{x=-\xi\cdot g(\xi)\}$$
		for some germ of non-vanishing holomorphic function $g$ at $\xi=0$ with real Taylor coefficients. Now, as $L_{\pm}$ are special Lagrangian submanifolds obtained by taking a hyper-K\"ahler rotation of holomorphic curves, 
they can be equipped with grading 0. Recall that $x=x_1-\sqrt{-1}x_2$ and $\xi=\xi_1+\sqrt{-1}\xi_2$, so
		\begin{align*}
			L_+=\{(\xi_1,\xi_2,\xi_1{\rm{Re}}(g)-\xi_2{\rm{Im}}(g),-\xi_2{\rm{Re}}(g)-\xi_1{\rm{Im}}(g))\in N_{\bb{R}}\times M_{\bb{R}}:(\xi_1,\xi_2)\in N_{\bb{R}}\},\\
			L_-=\{(\xi_1,\xi_2,-\xi_1{\rm{Re}}(g)+\xi_2{\rm{Im}}(g),\xi_2{\rm{Re}}(g)+\xi_1{\rm{Im}}(g))\in N_{\bb{R}}\times M_{\bb{R}}:(\xi_1,\xi_2) \in N_{\bb{R}}\}.
		\end{align*}
		Their tangent spaces at $(x,\xi)=0$ are given by
		\begin{align*}
			T_pL_+=\Span_{\bb{R}}\left\{\pd{}{\xi_1}+g(0)\pd{}{x_1},\pd{}{\xi_2}-g(0)\pd{}{x_2}\right\},\\
			T_pL_-=\Span_{\bb{R}}\left\{\pd{}{\xi_1}-g(0)\pd{}{x_1},\pd{}{\xi_2}+g(0)\pd{}{x_2}\right\}.
		\end{align*}
		In terms of standard complex coordinates $z_i:=\xi_i+\sqrt{-1}x_i$, we have
		\begin{align*}
			T_pL_+=\Span_{\bb{R}}\left\{
			\begin{pmatrix}
				1+\sqrt{-1}g(0)\\
				0
			\end{pmatrix},
			\begin{pmatrix}
				0\\
				1-\sqrt{-1}g(0)
			\end{pmatrix}\right\},\\
			T_pL_-=\Span_{\bb{R}}\left\{
			\begin{pmatrix}
				1-\sqrt{-1}g(0)\\
				0
			\end{pmatrix},
			\begin{pmatrix}
				0\\
				1+\sqrt{-1}g(0)
			\end{pmatrix}\right\}.
		\end{align*}
		Then the unitary matrix
		$$\frac{1}{\sqrt{1+g(0)^2}}\begin{pmatrix}
		0 & 1+\sqrt{-1}g(0)\\
		1-\sqrt{-1}g(0) & 0
		\end{pmatrix}$$
		carries $T_pL_+$ to $\bb{R}^2\subset\bb{C}$ and
		\begin{align*}
			\begin{pmatrix}
				1-\sqrt{-1}g(0)\\
				0
			\end{pmatrix}\mapsto\frac{1}{\sqrt{1+g(0)^2}}\begin{pmatrix}
				0\\
				1-g(0)^2-2\sqrt{-1}g(0)
			\end{pmatrix},\\
			\begin{pmatrix}
				0\\
				1+\sqrt{-1}g(0)
			\end{pmatrix}\mapsto\frac{1}{\sqrt{1+g(0)^2}}\begin{pmatrix}
				0\\
				1-g(0)^2+2\sqrt{-1}g(0)
			\end{pmatrix}.
		\end{align*}
		Hence the sum of the angles of intersection equals $\pi$ and this proves $\deg(p)=1$.
	\end{proof}
	
	Now, as $\bb{L}_{\varphi_{\bb{L}^{\trop}}}$ is a graded Lagrangian multi-section, we can give it the canonical grading. By Corollary \ref{cor:MS_gives_bundle}, its mirror is a rank 2 toric vector bundle $\cu{E}_{\bb{L}_{\varphi_{\bb{L}^{\trop}}}}$. The corresponding associated tropical Lagrangian multi-section is $\bb{L}^{\trop}$. Indeed, for any $\sigma\in\Sigma(2)$, we have
	$$(\cu{E}_{\bb{L}_{\varphi_{\bb{L}^{\trop}}}}|_{X_{\sigma}})_m\simeq\mu_{m,-\sigma}(\kappa(\cu{E}_{\bb{L}_{\varphi_{\bb{L}^{\trop}}}}))=\mu_{m,-\sigma}(\cu{F}_{\bb{L}_{\varphi_{\bb{L}^{\trop}}}})\simeq Hom_{\cu{F}uk_{\bb{K}}^0(Y)}(D_{m,-\sigma}[-n],\bb{L}_{\varphi_{\bb{L}^{\trop}}}).$$
	By the asymptotic condition $\bb{L}_{\varphi_{\bb{L}^{\trop}}}^{\infty}\subset\Lambda_{\bb{L}^{\trop}}^{\infty}$, the right hand side is non-zero if and only if $m=d\varphi^{\trop}|_{\Int(\sigma')}$ for some lift $\sigma'\in\Sigma_L(2)$ of $\sigma$. Hence $\bb{L}_{\cu{E}_{\bb{L}_{\varphi_{\bb{L}^{\trop}}}}}^{\trop}=\bb{L}^{\trop}$. Combining with Corollary \ref{cor:MS_gives_bundle} and Proposition \ref{prop:N>2}, we obtain the following beautiful corollary.

    \begin{corollary}
        When $N$ is odd, an $N$-generic tropical Lagrangian multi-section can be realized by an embedded Lagrangian multi-section if and only if $N\geq 3$.
    \end{corollary}

    Applying equivariant HMS to the Lagrangian multi-sections obtained in Theorem \ref{thm:unobs_immersed_Lag}, we obtain a toric-algebro-geometric result that is not known (at least to the authors) before. 
	
	\begin{corollary}\label{cor:existence_of_bundle}
		Suppose $\bb{L}^{\trop}$ is a $N$-generic 2-fold tropical Lagrangian multi-section over a complete fan $\Sigma$ on $N_{\bb{R}}\cong\bb{R}^2$ with $N\geq 3$. Then there is an indecomposable rank 2 toric vector bundle $\cu{E}$ on $X_{\Sigma}$ such that
		\begin{align*}
			\dim_{\bb{K}}\Ext_T^0(\cu{E},\cu{E})=&\,1,\\
			\dim_{\bb{K}}\Ext_T^1(\cu{E},\cu{E})=&\,N-3,\\
			\dim_{\bb{K}}\Ext_T^2(\cu{E},\cu{E})=&\,0,
		\end{align*}
		and $\bb{L}_{\cu{E}}^{\trop}=\bb{L}^{\trop}$.
	\end{corollary}

    \begin{remark}\label{rem:direct_sum}
		When $d$ is even, there is a distinguished degeneration of the hyper-elliptic curve
		$$
		x^2=a_{2g+2}\xi^2(\xi-r_1)^2\cdots(\xi-r_g)^2,
		$$ 
		with distinct $r_i$'s in $\bb{C}^{\times}$ by having a pair of $r_i$'s merge.
        Then $\bb{L}_{\varphi_{\bb{L}^{\trop}}}$ is actually given by a union of two Lagrangian sections (the trivial realization of $\bb{L}^{\trop}$) and is mirror to the direct sum of the line bundles corresponding to the two ``indecomposable" (see \cite{Suen_trop_lag} for the precise notion) components of $\bb{L}^{\trop}$.
	\end{remark}

 \subsection{Higher rank cases}

The higher rank realization problem is now no more difficult than the rank 2 cases. As every tropical Lagrangian multi-section can be decomposed into a union of maximal ones, it suffices to consider a maximal $r$-fold tropical Lagrangian multi-section $\bb{L}^{\trop}$ over a 2-dimensional complete fan $\Sigma$. With this assumption, the underlying $r$-fold covering map is topologically the map $p_{std}:z\mapsto z^r$ on $\bb{C}$. Hence the preimage of the circle $S^1\subset N_{\bb{R}}$ around the branch point under $p$ is a single circle $C\subset L^{\trop}$ and $Deck(C)\cong\bb{Z}/r\bb{Z}$. Parametrizing $C$ by $[0,2\pi)$. Then a deck transformation in $Deck(C)$ can be identified with the map $\theta\mapsto\theta+\frac{2\pi i}{r}\text{ mod }2\pi$ for some $i\in\bb{Z}_{\geq 0}$. In this case, the notion of $N$-genericity is generalized as follows.

\begin{definition}\label{def:N_generic_general_r}
    Let $\bb{L}^{\trop}$ be a maximal $r$-fold tropical Lagrangian multi-section over a complete 2-dimensional fan $\Sigma$. We say $\bb{L}^{\trop}$ is \emph{$N$-generic} if for any distinct deck transformations $\gamma_1,\gamma_2\in Deck(C)$, the graph of $\varphi^{\trop}\circ\gamma_1|_{[0,\frac{2\pi}{r})}$ intersects that of $\varphi^{\trop}\circ\gamma_2|_{[0,\frac{2\pi}{r})}$ transversely at exactly $N$ smooth points.
\end{definition}

We can now give a partial result for higher ranks.

\begin{theorem}\label{thm:LRP_higher_rank}
    A maximal $r$-fold tropical Lagrangian multi-section $\bb{L}^{\trop}$ over a complete 2-dimensional fan $\Sigma$ can be realized by an $r$-fold embedded (and hence unobstructed) Lagrangian multi-section if it is $\left\lfloor 2\left(\frac{d}{r}+1\right)\right\rfloor$-generic for some $d\in\bb{Z}_{>0}$ such that ${\rm{g.c.d.}}(r,d)=1$. In particular, such $\bb{L}^{\trop}$ can be realized by a rank $r$ toric vector bundle over $X_{\Sigma}$.
\end{theorem}
\begin{proof}[Sketch of proof]
The smoothing argument in Lemma \ref{lem:infinity_part} can be used to construct an $r$-fold Lagrangian multi-section $\bb{L}_{\varphi_{\geq R}}$ over $N_{\bb{R}}\backslash D_R$. To glue back the disk $D_R$, we use the embedded local model
$$L_{r,d}:=\{(x,\xi)\in\bb{C}^2:x^r=f_d(\xi)\},$$
where $f_d\in\bb{R}[\xi]$ is a polynomial of degree $d$ with leading coefficient $a_d>0$, $\text{g.c.d.}(r,d)=1$, and distinct roots. The underlying surface is embedded and has one cylindrical end. Projection onto the $\xi$-coordinate is an $r$-fold branched covering map. It is easy to write $L_{r,d}$ as
$$\{((p_{std}^*)^{-1}(d\varphi_{r,d}(l)),p_{std}(l))\in M_{\bb{R}}\times N_{\bb{R}}:l\in C_{\geq R}\},$$
where $\varphi_{\geq R}$ is given in terms of the polar coordinate $(r,\theta)$ on the cylindrical end $C_{\geq R}$ by
$$\varphi_{r,d}(r,\theta)=\sum_{i=0}^{\infty}c_kr^{d+r-ri}\cos((d+r-ri)\theta),$$
for $r\geq R>>1$ and some $c_i\in\bb{R}$ with $c_0>0$. By considering the leading order term of $\varphi_{r,d}$, we note that for any distinct $i,j\in\{0,1,\dots,r-1\}$, the equation
$$\varphi_{r,d}\left(r,\theta+\frac{2\pi i}{r}\right)=\varphi_{r,d}\left(r,\theta+\frac{2\pi j}{r}\right)$$
has $\left\lfloor 2\left(\frac{d}{r}+1\right)\right\rfloor$ solutions on $\left[0,\frac{2\pi}{r}\right)$. Then we can apply the previous glueing method to obtain an embedded $r$-fold Lagrangian multi-section $\bb{L}_{\varphi_{f_d}}$ that satisfies the asymptotic condition $\bb{L}_{\varphi_{f_d}}^{\infty}\subset\Lambda_{\bb{L}^{\trop}}^{\infty}$.
\end{proof}

\begin{remark}
    When $r=2$, we have $\left\lfloor 2\left(\frac{d}{2}+1\right)\right\rfloor=d+2\geq 3$, which is the necessary and sufficient condition for solving the $r=2$ Lagrangian realization problem. On the other hand, we call Theorem \ref{thm:LRP_higher_rank} a partial result because the $\left\lfloor 2\left(\frac{d}{r}+1\right)\right\rfloor$-genericity condition may not be necessary when $r\geq 3$.
\end{remark}

\begin{example}
    Consider the 3-fold tropical Lagrangian multi-section $\bb{L}^{\trop}$ as shown in Figure \ref{fig:trop_lag_rank3}
    \begin{figure}[H]
		\centering
		\includegraphics[width=110mm]{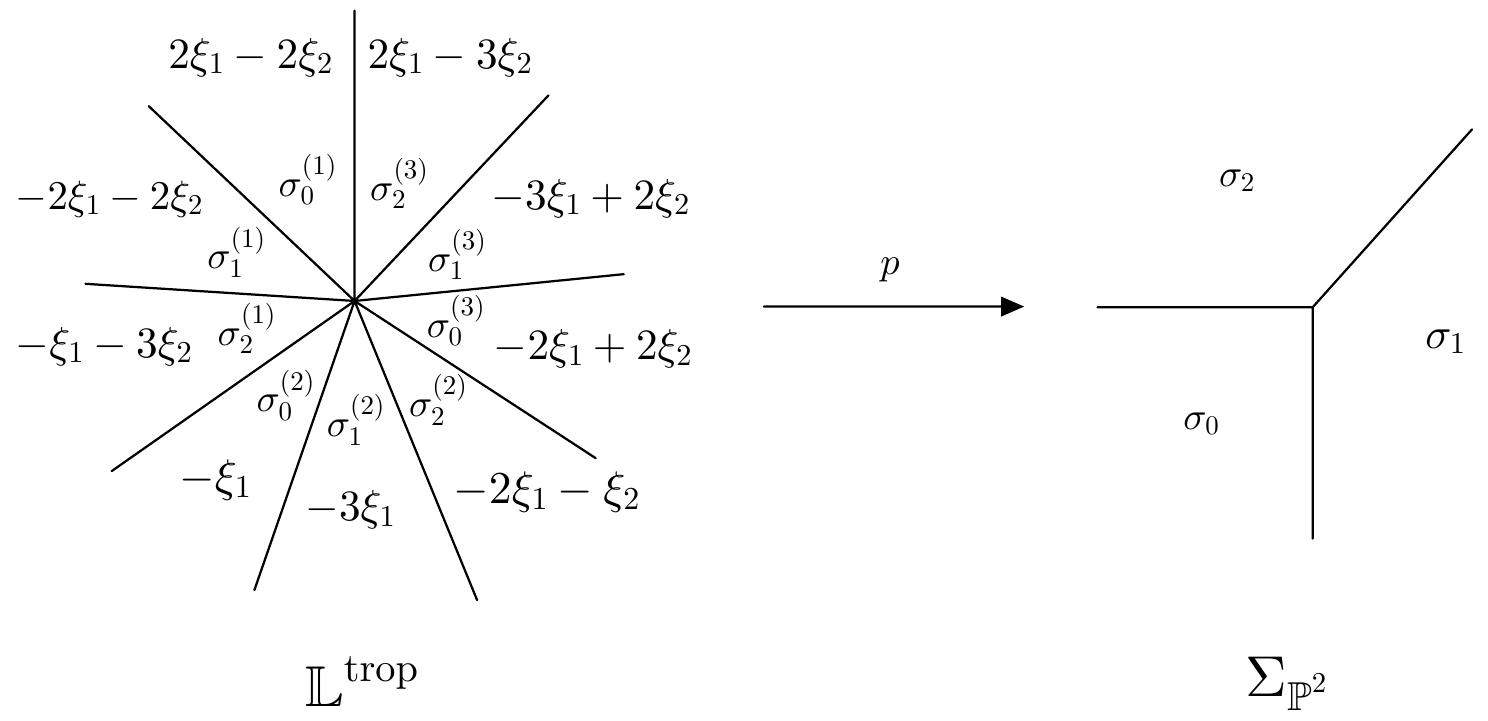}
		\caption{The 3-fold tropical Lagrangian multi-section $\bb{L}^{\trop}$ over $\Sigma_{\bb{P}^2}$.}
		\label{fig:trop_lag_rank3}
	\end{figure}
    One checks that
    \begin{align*}
        m(\sigma_0^{(1)})-m(\sigma_0^{(2)}),\,m(\sigma_0^{(2)})-m(\sigma_0^{(3)}),\,m(\sigma_0^{(1)})-m(\sigma_0^{(3)})\in(\sigma_0\cap\sigma_1)^{\vee}\cap M,\\
        m(\sigma_1^{(2)})-m(\sigma_1^{(1)}),\,m(\sigma_1^{(3)})-m(\sigma_1^{(2)}),\,m(\sigma_1^{(3)})-m(\sigma_1^{(1)})\in(\sigma_1\cap\sigma_2)^{\vee}\cap M,\\
        m(\sigma_2^{(1)})-m(\sigma_2^{(2)}),\,m(\sigma_2^{(2)})-m(\sigma_2^{(3)}),\,m(\sigma_2^{(1)})-m(\sigma_2^{(3)})\in(\sigma_2\cap\sigma_0)^{\vee}\cap M
    \end{align*}
    and none of them equal zero. Thus $\bb{L}^{\trop}$ is 3-generic, that is, $d=2$; hence, it is realizable. The underlying surface has one genus and one cylindrical end. The projection map $p_{N_{\bb{R}}}|_L:L\to N_{\bb{R}}$ is a  branched covering map with two degree 3 branch points, i.e. they locally look like $z\mapsto z^3$.
\end{example}
		
	\section{SYZ mirror to rank 2 toric vector bundles on $\bb{P}^2$}\label{sec:rk2}
	
	We now apply Theorem \ref{thm:unobs_immersed_Lag} to prove that every indecomposable\footnote{The case of decomposable rank 2 toric vector bundle is trivial as its mirror is the union of two Lagrangian sections.} rank 2 toric vector bundle on $\bb{P}^2$ has mirror being a Lagrangian multi-section.
	
	In \cite{Kaneyama_classification}, Kaneyama classified indecomposable rank 2 toric vector bundles on $\bb{P}^2$ by the following exact sequence
	$$0\to\cu{O}\to\cu{O}(aD_0)\oplus\cu{O}(bD_1)\oplus\cu{O}(cD_2)\to E_{a,b,c}\to 0,$$
	for $a,b,c\in\bb{Z}_{>0}$ and the first map is given by $1\mapsto(Z_0^a,Z_1^b,Z_2^c)$, where $(Z_i)$ denote the homogeneous coordinates of $\bb{P}^2$ and $D_i:=\{Z_i=0\}$. Then any indecomposable toric vector bundle on $\bb{P}^2$ is isomorphic to $E_{a,b,c}(D)$ or $E_{a,b,c}^*(D)$, for some integers $a,b,c>0$ and a toric divisor $D$.
	
	Let's focus on the case $D=0$. The associated tropical Lagrangian multi-section 
$$
\bb{L}_{a,b,c}^{\trop}:=(L^{\trop},\Sigma_L,\mu_L,p,\varphi_{a,b,c}^{\trop})
$$ 
of $E_{a,b,c}$ is depicted in Figure \ref{fig:L_abc_trop}.
	\begin{figure}[H]
		\centering
		\includegraphics[width=110mm]{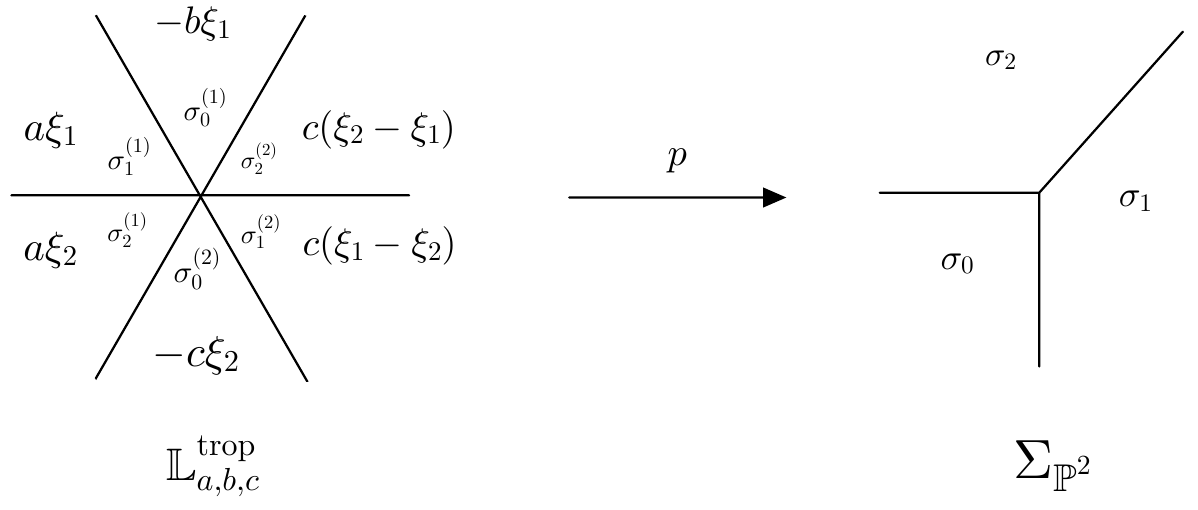}
		\caption{The associated tropical Lagrangian multi-section $\bb{L}_{a,b,c}^{\trop}$ of $E_{a,b,c}$. The domain $L$ is homeomorphic to $\bb{C}$ and the branched covering map $p:L\to N_{\bb{R}}$ can be topologically identified with the square map $z\mapsto z^2$ on $\bb{C}$.}
		\label{fig:L_abc_trop}
	\end{figure}
	It's not hard to see that $\bb{L}_{a,b,c}^{\trop}$ is $3$-generic. It can therefore be realized by a Lagrangian multi-section $\bb{L}_{a,b,c}$. In this case, as $k=1$, we have $d=1$ and $g=0$. Hence $\bb{L}_{a,b,c}$ is diffeomorphic to $\bb{R}^2$. Let $\cu{E}_{\bb{L}_{a,b,c}}$ be the mirror of $\bb{L}_{a,b,c}$. We show that $\cu{E}_{\bb{L}_{a,b,c}}\cong E_{a,b,c}$. The following lemma shows that an indecomposable rank 2 toric vector bundle on $\bb{P}^2$ is equivariantly rigid.
	
	\begin{lemma}\label{lem:equiv_rigid}
		For any $a,b,c\in\bb{Z}_{>0}$ and toric divisor $D$, $E_{a,b,c}(D)$ and $E_{a,b,c}^*(D)$ are equivariantly rigid.
	\end{lemma}
	\begin{proof}
		Suppose $E_{a,b,c}(D)$ and $E_{a',b',c'}(D')$ share the same associated tropical Lagrangian multi-section. Then $\bb{L}_{a,b,c}=\bb{L}_{a',b',c'}(D'-D)$. It thus suffices to consider the case $D=0$. Write $D'=k_0D_0+K_1D_1+k_2D_2$. In this case, by comparing the associated tropical Lagrangian section of $E_{a,b,c}$ and $E_{a',b',c'}(D')$, one can easily show that $k_i=0$ for all $i=0,1,2$. Hence we must have $a'=a,b'=b,c'=c$. The same argument applies to $E_{a,b,c}^*$ which has the tropical Lagrangian multi-section $-\bb{L}_{a,b,c}$.
	\end{proof}
	
	As $\bb{L}_{\cu{E}_{\bb{L}_{a,b,c}}}^{\trop}=\bb{L}_{a,b,c}^{\trop}$, we have $\cu{E}_{\bb{L}_{a,b,c}}\cong E_{a,b,c}$ by Lemma \ref{lem:equiv_rigid}. We immediately deduce the following theorem.
	
	\begin{theorem}\label{thm:TGRC_E_abc}
		For $a,b,c\in\bb{Z}_{>0}$, the mirror of $E_{a,b,c}$ is quasi-isomorphic to an embedded, simply connected, and canonically graded Lagrangian multi-section $\bb{L}_{a,b,c}\subset Y$ so that $\bb{L}_{a,b,c}^{\infty}\subset\Lambda_{a,b,c}^{\infty}$.
	\end{theorem}
	
	Now, for the case $E_{a,b,c}^*$, we have the corresponding conical Lagrangian subset
	$$\Lambda_{-\bb{L}_{a,b,c}^{\trop}}:=\bigcup_{\tau'\in\Sigma_L}-m(\tau')\times-\tau.$$
	It is then easy to see that
	$$-\bb{L}_{a,b,c}:=\{(-(p_{std}^*)^{-1}(d\til{\varphi}_{a,b,c}(l)),-p_{std}(l))\in Y:l\in\bb{R}^2\}$$
	is the mirror Lagrangian multi-section of $E_{a,b,c}^*$, which is also embedded. For a toric divisor $D$, let $\varphi_D:N_{\bb{R}}\to\bb{R}$ be a smoothing of $\varphi_D^{\trop}$ that asymptotic to the slopes of $\varphi_D^{\trop}$, the associated piecewise linear function of $D$. We have
	$$\Lambda_{\bb{L}_{a,b,c}^{\trop}(D)}:=\bigcup_{\tau'\in\Sigma_L}(m(\tau')+m_D(\tau))\times-\tau,$$
	where $m_D(\tau)$ is the slope of $\varphi_D^{\trop}$ along $\tau$. We define
	$$\bb{L}_{a,b,c}(D):=\{((p_{std}^*)^{-1}(d(\til{\varphi}_{a,b,c}+\varphi_D\circ p_{std})(l)),-p_{std}(l))\in Y:l\in\bb{R}^2\},$$
	which is again embedded as one can easily check. As $E_{a,b,c}(D)$ is also equivariantly rigid, meaning that it admits no non-trivial deformations as toric vector bundles, the same argument of the proof of Theorem 4.4 allows us to conclude that $\bb{L}_{a,b,c}(D)$ is mirror to $E_{a,b,c}(D)$. Therefore, our result can be extended to arbitrary rank 2 indecomposable toric vector bundles on $\bb{P}^2$.
	
	\begin{theorem}\label{thm:general_case}
		The mirror of a rank 2 indecomposable toric vector bundle $\cu{E}$ on $\bb{P}^2$ is quasi-isomorphic to an embedded, simply connected, and canonically graded Lagrangian multi-section $\bb{L}_{\cu{E}}\subset Y$ so that $\bb{L}_{\cu{E}}^{\infty}\subset\Lambda_{\bb{L}_{\cu{E}}^{\trop}}^{\infty}$.
	\end{theorem}
	
	\begin{remark}
		The rank 2 toric vector bundle $E_{m,n}$ over $\bb{P}^2$ considered in \cite{CMS_k3bundle} is actually isomorphic to $E_{m-n,m-n,m-n}((2n-m)D_0)$ and so it has mirror $\bb{L}_{m-n,m-n,m-n}((2n-m)D_0)$.
	\end{remark}

	\begin{remark}
		Let $\ol{Y}:=Y/M$ be the obvious quotient. Theorem \ref{thm:general_case} and the non-equivariant HMS for $\bb{P}^2$ suggest that the non-equivariant mirror of a rank 2 toric vector bundle $\cu{E}$ is given by the quotient $\ol{\bb{L}}_{\cu{E}}\subset\ol{Y}$, which can be immersed (after Hamiltonian perturbation) in general. As $\Ext^1_T(\cu{E},\cu{E})=0$ when $\cu{E}$ is indecomposable, we see that the immersed sector of $\ol{\bb{L}}_{\cu{E}}$ are responsible for the whole $\Ext^1(\cu{E},\cu{E})$, which can be arbitrarily large as $a,b,c\to\infty$.
	\end{remark}

\appendix

\section{Nadler's generation result for the immersed Fukaya category}
\label{appendix:nadler}

In this section, we provide some more explanation on how Nadler's proof
in \cite[Section 4]{Nadler} applies to the case of \emph{tautologically unobstructed} immersed Lagrangians with clean self-intersection in the cotangent bundle $T^*X$ for a general compact manifold $X$. To fit in our context, for a finite collection of Lagrangian immersions in $T^*M_{\bb{R}}$ with compact support in the $M_{\bb{R}}$-direction, we regard them as Lagrangian immersions in $T^*M_{\bb{R}}/\varepsilon^{-1}M$, for some small $\varepsilon>0$, so that no extra immersed points are created under the quotient. Their Floer theory are, by definition, the one defined in $T^*M_{\bb{R}}/\varepsilon^{-1}M$. See \cite{CCC_HMS}.

Let $\cu{F}uk(T^*X)$ be the infinitesimally wrapped tautologically unobstructed immersed Fukaya category, whose objects are tautologically unobstructed Lagrangian immersions with clean self-intersection and Legendrian boundary condition on $S^*X$. The main argument in Nadler's proof is to express the Yoneda image 
$$
\cu{Y}_{\bb{L}}: = \Hom_{\cu{F}uk(T^*X)}(\cdot,\bb{L})
$$
as a twisted complex with terms $\Hom_{\cu{F}uk(T^*X)}(\alpha_X(\cdot), L_{\tau_{\mathfrak a}*})$
(See \cite[Section 4.5]{Nadler}.), where $\alpha_X$ is the involution $(x,\xi)\mapsto(x-\xi)$. We summarize the key ingredients of Nadler's argument into three lemmas. For a general submanifold $Y\subset X$ and a smooth function $m:X\to\bb{R}_{\geq 0}$ whose zero set is given by $\ol{Y}\backslash Y$, define the corresponding standard brane $L_{Y,m*}$ by
$$L_{Y,m*}:=N_{Y/X}^*+\Gamma_{d\log(m)}\subset T^*X,$$
where $N_{Y/X}^*$ is the conormal bundle of $Y$. Different choices of $m$ lead to isomorphic brane \cite[Theorem 7.0.3]{NZ}, so we write $L_{Y*}$ for the standard brane associated to $Y$. The diagonal brane
$$
L_{\Delta_X*}:=\{(x,\xi,x,-\xi)\in T^*X\times T^*X\},
$$
is the standard brane associated to the diagonal $\Delta_X\subset X\times X$. We also take a triangulation $\cu{T}$ of $X$ and denote by $\tau_a\in\cu{T}$ a cell thereof. For $\tau_a\in\cu{T}$ and $\Delta_{\tau_a}\subset X\times X$, we have the corresponding standard branes
$L_{{\tau_a}*}\subset T^*X$ and $L_{\Delta_{\tau_a}*}\subset T^*(X\times X)\cong T^*X\times T^*X$. Finally, for a point $x_a\in\Int(\tau_a)$, the standard brane $L_{\{x_a\}*}$ is nothing but the cotangent fibre of $T^*X$ at $x_a$.

\begin{lemma}\label{lem:A1}
    For any $\bb{L}_1,\bb{P}_1\in\cu{F}uk(T^*X_1)$ and $\bb{L}_2,\bb{P}_2\in\cu{F}uk(T^*X_2)$, there exists a 
    quasi-isomorphism
    $$\Hom_{\cu{F}uk(T^*X_1\times T^*X_2)}(\bb{L}_1\times\bb{P}_1,\bb{L}_2\times\bb{P}_2)\simeq \Hom_{\cu{F}uk(T^*X_1)}(\bb{L}_1,\bb{P}_1)\otimes \Hom_{\cu{F}uk(T^*X_2)}(\bb{L}_2,\bb{P}_2).$$
\end{lemma}

%\begin{lemma}\label{lem:A2}
%    For any $\bb{L}_1,\bb{L}_2\in\cu{F}uk(T^*X)$, there is an quasi-isomorphism
%    $$hom_{\mathcal{F}(T^*X)}(\bb{L}_1,\bb{L}_2)\simeq hom_{\mathcal{F}(T^*X\times T^*X)}(L_{\Delta_X},\bb{L}_2\times\alpha_X(\bb{L}_1))$$
%\end{lemma}

\begin{lemma}\label{lem:A3}
    The brane $\alpha_{X\times X}(L_{\Delta_X*})\subset T^*(X\times X)$ can be written as a twisted complex of standard objects of the form $L_{\Delta_{\tau_a}*}$, $\tau_a\in\cu{T}$.
\end{lemma}

\begin{lemma}\label{lem:A4}
    For any $\bb{L}_1,\bb{L}_2\in\cu{F}uk(T^*X)$ and a fine enough triangulation $\cu{T}$ of $X$ such that $L_1^{\infty},L_2^{\infty}\subset\Lambda_{\cu{T}}^{\infty}$, there is a quasi-isomorphism
    $$
    \Hom_{\cu{F}uk(T^*X\times T^*X)}(\bb{L}_1\times\bb{L}_2,L_{\Delta_{\tau_a}*})\simeq \Hom_{\cu{F}uk(T^*X)}(\bb{L}_1,L_{\{x_a\}*})\otimes \Hom_{\cu{F}uk(T^*X)}(\bb{L}_2,L_{\tau_a*})
    $$
for each cell $\tau_a\in\cu{T}$.
\end{lemma}
Assuming these lemmas, for any object $\bb{L}\in\cu{F}uk(T^*X)$ and test object $\bb{P}\in\cu{F}uk(T^*X)$, we have
$$\Hom_{\cu{F}uk(T^*X)}(\bb{P},\bb{L})\simeq\left(\bigoplus_{\alpha}\Hom_{\cu{F}uk(T^*X)}(\alpha_X(\bb{L}),L_{\{x_{a_{\alpha}}\}*})\otimes \Hom_{\cu{F}uk(T^*X)}(\bb{P},L_{\tau_{a_{\alpha}}*}),(\delta_{\alpha\beta})_{\alpha<\beta}\right),
$$
i.e., $\Hom_{\cu{F}uk(T^*X)}(\bb{P},\bb{L})$ can be expressed as a twisted complex with 
terms the Floer complexes
$$\Hom_{\cu{F}uk(T^*X)}(\alpha_X(\bb{L}), L_{\tau_{a_{\alpha}}})$$
for some
collection $\{a_{\alpha}\}_\alpha$. This shows that the microlocal functor $Sh_{cc}(X)\to\cu{F}uk(T^*X)$ is essentially surjective.

We now give reasons for the validity of each lemma.  Once Lemma \ref{lem:A1} is established, Lemma \ref{lem:A3} and \ref{lem:A4} apply verbatim without any change from Nadler's original paper 
\cite[Section 4.3 and Proposition 4.4.1]{Nadler}.
Indeed, Lemma \ref{lem:A1} follows from the K\"unneth formula provided in \cite[Corollary 16.10]{fukaya_immersed}: The only thing we need to make sure is that 
the product $\bL_1 \times \bL_2$ of two objects be an object of the $\cu{F}uk(T^*X_1 \times T^*X_2)$ 
when each $\bL_i$ is an object of $\cu{F}uk(T^*X_i)$ for each $i = 1,2$ respectively. 
Since the product  of two immersed (non-embedded) Lagrangians do
not have transverse self-intersections but only clean self-intersections of higher dimensions, the product
$\bL \times \bP$  do not satisfy the transversal self-intersection hypothesis imposed in 
Definition \ref{defn:lag-immersion}. Because of this, to ensure 
Lemma \ref{lem:A1}, we need to employ Fukaya's framework of immersed Lagrangian Floer theory whose objects consist of immersed Lagrangian submanifolds \emph{with clean self-intersections}
\cite[Theorem 16.9]{fukaya_immersed}.  A full explanation on how this Fukaya's enlarged setting 
 of immersed Lagrangian Floer theory in \cite{fukaya_immersed} gives rise to the proof of Lemma \ref{lem:A1}
 will be given in \cite{OS_generation} while we perform a microlocalization of 
the immersed Lagrangian Floer theory over the Novikov field $\bb{K}$.

\bibliographystyle{amsalpha}
\bibliography{geometry-oh}

\end{document}